\documentclass[10pt]{amsart}

\pdfoutput=1

\usepackage{amsmath,amssymb,graphicx,bbm}
\usepackage{amsthm,verbatim}
\usepackage{mathrsfs,mathtools}
\usepackage{mathabx,algorithm}
\usepackage{url,hyperref}

\usepackage[footnotesize,bf]{caption}
\usepackage[left=1.25in,right=1.25in,top=1in]{geometry}

\usepackage[algo2e,norelsize,lined,boxed,linesnumbered,commentsnumbered]{algorithm2e}%This option changes the name of the environment algorithm from the algorithm2e package into algorithm2e and so avoids the conflict with the package which already define an algorithm environment (algorithm) is required by siamart

\usepackage{mathdefs,xcolor}
\usepackage{array,multirow,booktabs} % for fancier tables

\newcommand{\bs}[1]{\boldsymbol{#1}}

\newcommand{\supp}{\textrm{supp}}
\renewcommand{\tilde}[1]{\widetilde{#1}}

\newlength{\leftstackrelawd}
\newlength{\leftstackrelbwd}
\def\leftstackrel#1#2{\settowidth{\leftstackrelawd}%
{${{}^{#1}}$}\settowidth{\leftstackrelbwd}{$#2$}%
\addtolength{\leftstackrelawd}{-\leftstackrelbwd}%
\leavevmode\ifthenelse{\lengthtest{\leftstackrelawd>0pt}}%
{\kern-.5\leftstackrelawd}{}\mathrel{\mathop{#2}\limits^{#1}}}

\makeatletter
\def\@tempa#1{\@xp\@tempb\meaning#1\@nil#1}
\def\@tempb#1>#2#3 #4\@nil#5{%
  \@xp\ifx\csname#3\endcsname\mathaccent
    \@tempc#4?"7777\@nil#5%
  \else
    \PackageWarningNoLine{amsmath}{%
      Unable to redefine math accent \string#5}%
  \fi
}
\def\@tempc#1"#2#3#4#5#6\@nil#7{%
  \chardef\@tempd="#3\relax\set@mathaccent\@tempd{#7}{#2}{#4#5}}
\@tempa\widetilde
\makeatother

\title{Computation of Induced Orthogonal Polynomial Distributions}
\author{Akil Narayan}
\thanks{Department of Mathematics, and Scientific Computing and Imaging (SCI) Institute, University of Utah, USA. {\tt akil@sci.utah.edu}. A. Narayan is partially supported by NSF DMS-1552238, AFOSR FA9550-15-1-0467, and DARPA EQUiPS N660011524053.}

\begin{document}

\maketitle
\begin{abstract}
  We provide a robust and general algorithm for computing distribution functions associated to induced orthogonal polynomial measures. We leverage several tools for orthogonal polynomials to provide a spectrally-accurate method for a broad class of measures, which is stable for polynomial degrees up to at least degree 1000. Paired with other standard tools such as a numerical root-finding algorithm and inverse transform sampling, this provides a methodology for generating random samples from an induced orthogonal polynomial measure. Generating samples from this measure is one ingredient in optimal numerical methods for certain types of multivariate polynomial approximation. For example, sampling from induced distributions for weighted discrete least-squares approximation has recently been shown to yield convergence guarantees with a minimal number of samples. We also provide publicly-available code that implements the algorithms in this paper for sampling from induced distributions.
\end{abstract}

%\begin{keywords}
%  Orthogonal polynomials, induced distributions, sampling
%\end{keywords}
%
%\begin{AMS}
%  asdf
%\end{AMS}

\section{Introduction}

Let $\mu$ be a probability measure on $\R$ such that a family of $L^2_{\dx{\mu}}$-orthonormal polynomials $\left\{p_n\right\}_{n=0}^\infty$ can be defined.\footnote{See Section \ref{sec:background:ops} for technical conditions implying this.} The non-decreasing function
\begin{align*}
  F_n(x) &= \int_{-\infty}^x p_n^2(t) \dx{\mu}(t), & x &\in \R.
\end{align*}
is a probability distribution function on $\R$ since $p^2_n$ has unit $\mu$-integral over $\R$. This paper is chiefly concerned with developing algorithms for drawing random samples of a random variable whose cumulative distribution function is $F_n$. The high-level algorithmic idea is straightforward: develop robust algorithms for evaluating $F_n$, and subsequently use a standard root-finding approach to compute $F_n^{-1}(U)$ where $U$ is a continuous uniform random variable on $[0,1]$. (This is colloquially called ``inverse transform sampling".) The challenge that this paper addresses is in the computational evaluation of $F_n(x)$ for any $n \in \N_0$ and for relatively general $\mu$. Borrowing terminology from \cite{gautschi_set_1993}, we call $F_n$ the order-$n$ distribution \textit{induced} by $\mu$. 

In our algorithmic development, we focus on three classes of continuous measures $\mu$ from which induced distributions spring: (1) Jacobi distributions on $[-1,1]$, (2) Freud (i.e., exponential) distributions on $\R$, and (3) ``Half-line" Freud distributions on $[0, \infty)$. These measures encompass a relatively broad selection of continuous measures $\mu$ on $\R$. %We briefly consider a fourth case of discrete measures $\mu$, in which case a discrete measure, a relatively straightforward algorithm for evaluating $F_n(x)$ can be described. See Section ???.

  The utility of sampling from univariate induced distributions has recently come into light: The authors in various papers \cite{hampton_coherence_2015,narayan_christoffel_2016,cohen_optimal_2016} note that additive mixtures of induced distributions are optimal sampling distributions for constructing multivariate polynomial approximations of functions using weighted discrete least-squares from independent and identically-distributed random samples. ``Optimal" means that these distributions define a sampling strategy which provides stability and accuracy guarantees with a sample complexity that is currently thought to be the best (smallest). This distribution also arises in related settings \cite{jakeman_generalized_2016}. The ability to sample from an induced distribution, which this paper addresses, therefore has significant importance for multivariate applied approximation problems. 

  Induced distributions can also help provide insight for more theoretical problems. The weighted pluripotential equilibrium measure is a multivariate probability measure that describes asymptotic distributions of optimal sampling points \cite{bloom_convergence_2010,narayan_christoffel_2016}. However, an explicit form for this measure is not known in general. The authors in \cite{narayan_christoffel_2016} make conjectures about the Lebesgue density associated to equilibrium measure in one case when its explicit form is currently unknown. While these conjectures remain unproven, univariate induced distributions can be used to simulate samples from the equilibrium measure. Hence, sampling from induced distributions can be used to provide supporting evidence for the theoretical conjectures in \cite{narayan_christoffel_2016}. 
  
  The outline of this paper is as follows: in Section \ref{sec:background} we review many standard properties of general orthogonal polynomial systems that are exploited for computing induced distributions. Section \ref{sec:Fn-eval} contains a detailed discussion of our novel approach for computing $F_n(x)$ for three classes of measures; this section also utilizes potential theory results in order to approximate $F_n^{-1}(0.5)$. Section \ref{sec:invert-induced} uses the previous section's algorithms in order to formulate an algorithmic strategy for computing $F_n^{-1}(u)$, $u \in [0, 1]$. Finally, Section \ref{sec:applications} discusses the above-mentioned applications of multivariate polynomial approximation using discrete least-squares, and investigating conjectures for a weighted equilibrium measure.
%This application of induced distributions is described in Section \ref{sec:applications-equilibrium}.

  Code that reproduces many of the plots in this paper is available for download \cite{narayan_akilnarayan/induced-distributions}. The code contains routines for accomplishing almost all of the procedures in this paper including evaluation and inversion of induced distributions (for many of the distributions in Table \ref{tab:measures}), inverse transform sampling for multivariate sampling from additive mixtures of induced distributions, and fast versions of all codes that utilize approximate monotone spline interpolants for fast evaluation and inversion of distribution functions. The code also contains routines that reproduce Figure \ref{fig:jacobi-distribution} (left, center), Figure \ref{fig:algorithm} (right), Figure \ref{fig:jacobi-error} (right), Figure \ref{fig:half-freud-mrs} (left), Figure \ref{fig:half-freud-error} (left), and Figure \ref{fig:equilibrium} (left). 

\begin{table}
  \begin{center}
  \resizebox{\textwidth}{!}{
    \renewcommand{\tabcolsep}{0.2cm}
    \renewcommand{\arraystretch}{1.3}
    {\scriptsize
    %\begin{tabular}{@{}cp{0.8\textwidth}@{}}
    \begin{tabular}{@{}lllllp{0.3\textwidth}@{}}
      \toprule
      Jacobi & $\dx{\mu}^{(\alpha,\beta)}_J(x) = \frac{1}{c_J^{(\alpha,\beta)}} (1-x)^\alpha (1+x)^\beta$ & $x \in [-1,1]$ & $\alpha > -1$ & $\beta > -1$ & $c_J^{(\alpha,\beta)} = 2^{\alpha + \beta + 1} B(\beta + 1, \alpha + 1)$ \\
      Half-line Freud & $\dx{\mu}^{(\alpha,\rho)}_{HF}(x) = \frac{1}{c_{HF}^{(\alpha,\rho)}} x^\rho \exp\left(-x^\alpha\right)$ & $x \geq 0$ & $\alpha > \frac{1}{2}$ & $\rho > -1$ & $c_{HF}^{(\alpha,\rho)} = \frac{1}{\alpha} \Gamma \left(\frac{\rho+1}{\alpha}\right)$ \\
      Freud & $\dx{\mu}^{(\alpha,\rho)}_{F}(x) = \frac{1}{c_{F}^{(\alpha,\rho)}} |x|^\rho \exp\left(-|x|^\alpha\right)$ & $x \in \R$ & $\alpha > 1$ & $\rho > -1$ & $c_{F}^{(\alpha,\rho)} = \frac{2}{\alpha} \Gamma \left(\frac{\rho+1}{\alpha}\right)$ \\
    \bottomrule
    \end{tabular}
  }
    \renewcommand{\arraystretch}{1}
    \renewcommand{\tabcolsep}{12pt}
  }
  \end{center}
  \caption{Classes of measures considered in this paper. $\Gamma(\cdot)$ is the Euler Gamma function, and $B(\cdot)$ is the Beta function.}\label{tab:measures}
\end{table}

\subsection{A simple example}

The main algorithmic novelties of this paper revolve around evaluation of $F_n(x)$. %However, clearly such an evalution is simply a univariate integral (albeit with oscillatory integrand). 
In Figure \ref{fig:jacobi-distribution} we show one example of the integrand $p_n^2(x) \dx{\mu}(x)$ and the associated $F_n$. One suspects that packaged integration routines should be able to perform relatively well in order to compute integrals for such a problem. In our experience this is frequently true, but comes at a price of increased computational effort and time, and decreased robustness. The right-hand pane in Figure \ref{fig:jacobi-distribution} shows timings for Matlab's built-in \texttt{integral} routine versus the algorithms developed in this paper. We see that the algorithms in this paper are much faster, usually resulting in around an order of magnitude savings.

\begin{figure}
  \begin{center}
    \includegraphics[width=0.32\textwidth]{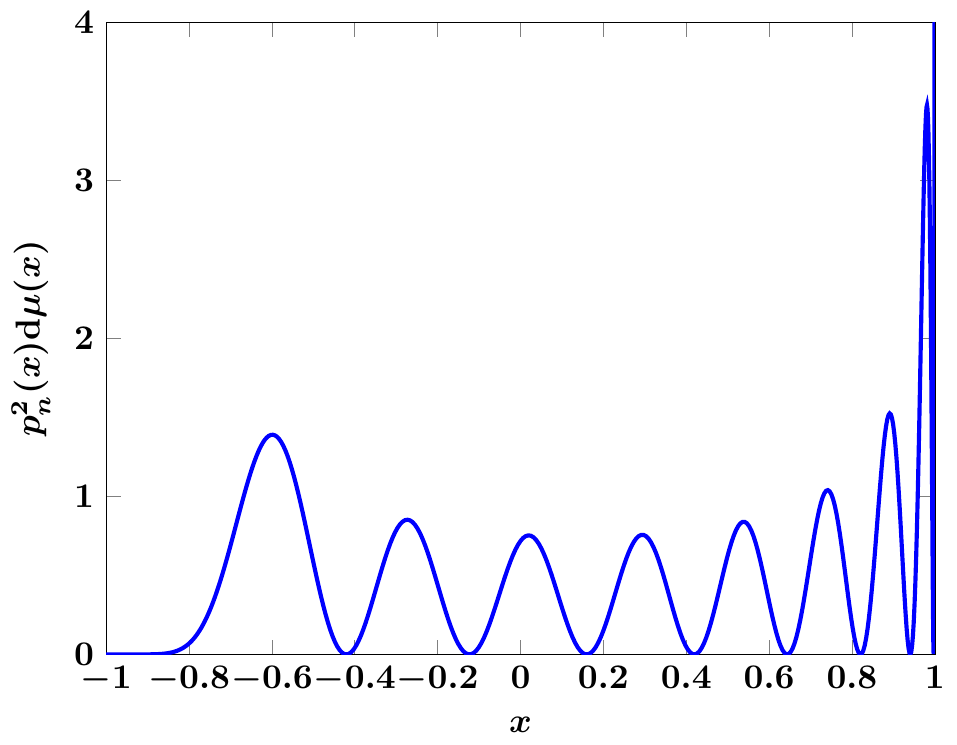}
    \includegraphics[width=0.32\textwidth]{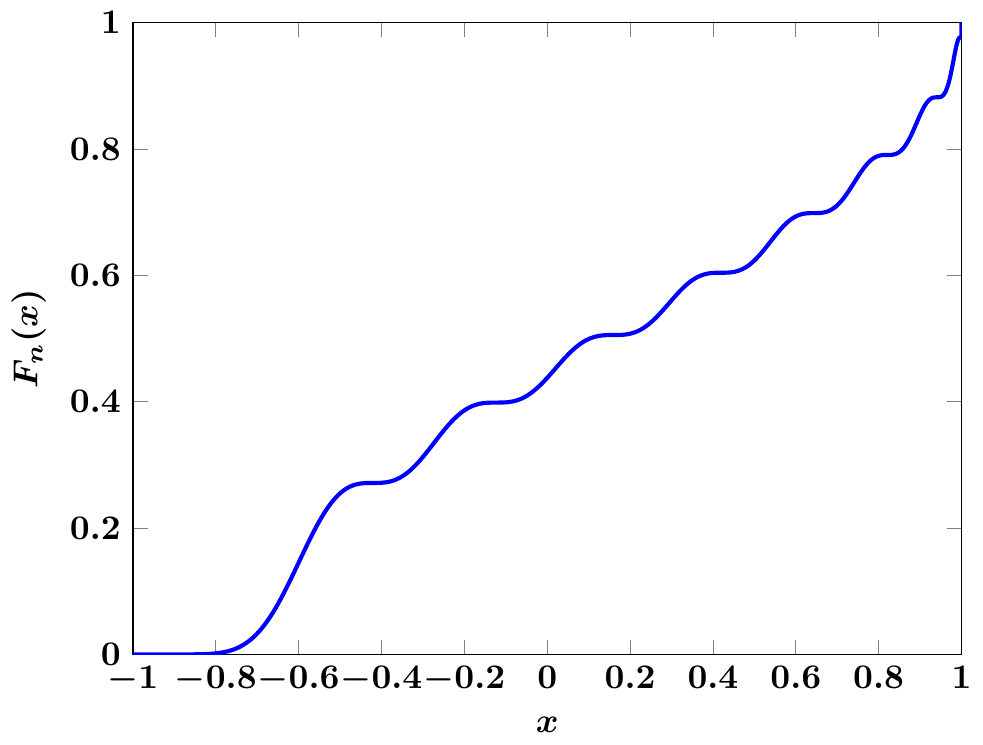}
    \includegraphics[width=0.32\textwidth]{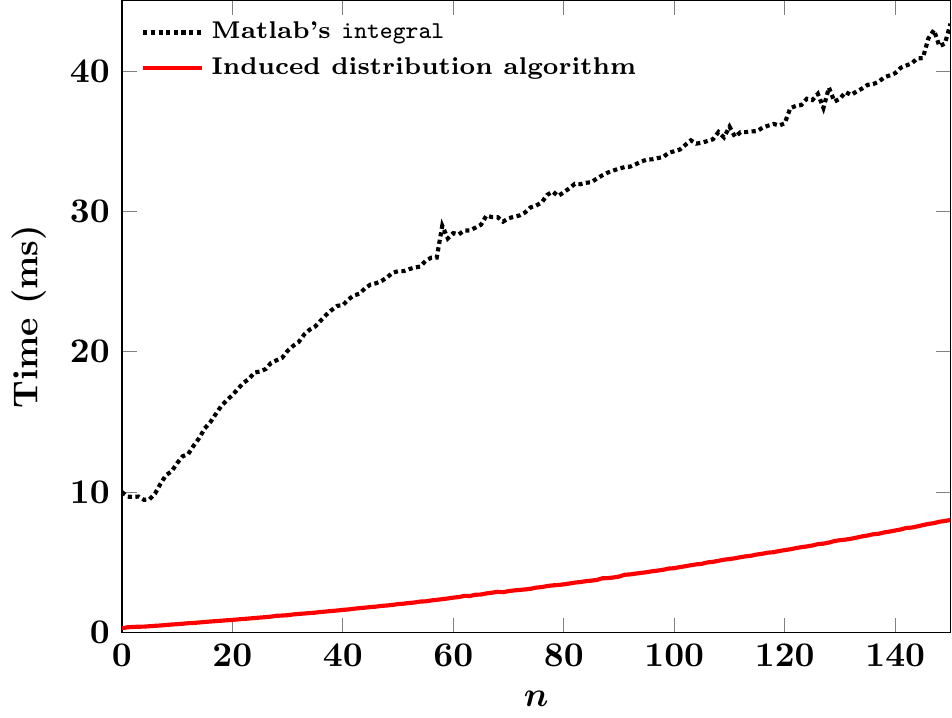}
  \end{center}
  \caption{Left: The induced distribution integrand $p_n^2(x) \dx{\mu}(x)$ for the Jacobi measure $\mu^{(\alpha,\beta)}_J$ in Table \ref{tab:measures} with parameters $\alpha = -0.8$, $\beta = \sqrt{101}$, and $n=13$. Center: The associated distribution function $F_n(x)$. Right: Computational timings for evaluation of $F_n(\cdot)$ (averaged over 100 runs) using Matlab's built-in \texttt{integral} routine compared to the algorithms developed in this paper. Timings were performed in Matlab (2015b) on a single-core 1.7 GHz Intel i5 processor with 4GB of RAM.}\label{fig:jacobi-distribution}
\end{figure}

Our experimentation (using Matlab's \texttt{integral}) also reveals the following advantages of using the specialized algorithm in this paper:
\begin{itemize}
  \item The results from our algorithm appear to be more accurate compared to standard routines, and use significantly less computation. This conclusion is based on our testing with \texttt{integral}, and is true even if one modifies algorithm tolerances in \texttt{integral}.%We set parameters in our algorithm to ``overkill" the discretization, yet our algorithm is faster and appears to be more accurate by about 4 digits. When we require \texttt{integral} to be as accurate as our algorithm, the timings for \textit{integral} in Figure \ref{fig:jacobi-distribution} are even slower.
  \item We have occasionally observed \texttt{integral} return non-monotonic evaluations for the distribution $F_n$. This typically happens when most of the mass of the integrand is concentrated far from the boundary of the support of $\mu$. Non-monotonic behavior causes problems in performing inverse transform sampling.
  \item When $\dx{\mu}(x)$ has a singularity at boundaries of the support of $\mu$, then \texttt{integral} frequently complains about singularities and failure to achieve error tolerances.
\end{itemize}
%One could likely specialize \texttt{integral} to address all of the difficulties above, but if one is going to such trouble to specialize an algorithm, we argue that it's better to use the specialized algorithm we have developed, especially since our algorithm appears to be notably faster.

\subsection{Historical discussion}
The distribution function of the arcsine or ``Chebyshev" measure is
\begin{align*}
  F(x) &= \frac{1}{\pi} \int_{-1}^x \frac{1}{\sqrt{1 - t^2}} \dx{t} = \frac{1}{2} + \frac{1}{\pi} \arcsin(x), & x &\in [-1,1].
\end{align*}
It was shown in \cite{nevai_orthogonal_1980} that if $\mu$ belongs to the Nevai class of measures, then $F_n(x) \rightarrow F(x)$ pointwise for all $x \in [-1,1]$. Further refinements on this statement were made in \cite{simon_ratio_2004} which generalized the class of measures for which convergence holds. Generating polynomials orthogonal with respect to the measure associated to $F_n$ is considered in \cite{gautschi_set_1993} with a generalization given in \cite{li_construction_1999}. 

The authors in \cite{hampton_coherence_2015} proposed sampling from an additive mixture of induced distributions using a Markov Chain Monte Carlo method for the purposes of computing polynomial approximations of functions via discrete least-squares; the work in \cite{narayan_christoffel_2016} investigates sampling from the $n$-asymptotic limit of these additive mixtures. The authors in \cite{cohen_optimal_2016} leverage the additive mixture property to sample from this distribution using a monotone spline interpolant. At the very least this latter method requires multiple evaluations of $F_n(x)$. To our knowledge there has been essentially no investigation into robust algorithms for the evaluation of $F_n$ for broad classes of measures, which is the subject of this paper. 

\section{Background}\label{sec:background}

\subsection{Orthogonal polynomials}\label{sec:background:ops}
This section contains classical knowledge, most of which is available from any seminal reference on orthogonal polynomials \cite{szego_orthogonal_1975,freud_orthogonal_1971,nevai_orthogonal_1980,gautschi_orthogonal_2004}.

Let $\mu$ be a Lebesgue-Stiltjies probability measure on $\R$, i.e., the distribution function 
\begin{align}\label{eq:F-definition}
  F(x) = \int_{-\infty}^x \dx{\mu}(t),
\end{align}
is non-decreasing and right-continuous on $\R$, with $F(-\infty) = 0$ and $F(\infty) = 1$. For any distribution $F$ we use the notation $F^c(x) \coloneqq 1 - F(x)$ for its complementary function.

We assume that $\mu$ has in infinite number of points of increase, and has finite polynomial moments of all orders, i.e.,
\begin{align*}
  \left|\int_\R x^n \dx{\mu}(x)\right| &< \infty, & n = 0, 1, \ldots.
\end{align*}
Under these assumptions, a sequence of orthonormal polynomials $\left\{p_n \right\}_{n=0}^\infty$ exists, with $\deg p_j = j$, satisfying
\begin{align*}
  \int_{\R} p_j(x) p_k(x) \dx{\mu(x)} = \delta_{k,j},
\end{align*}
where $\delta_{k,j}$ is the Kronecker delta function. We will write $p_n(\cdot) = p_n(\cdot; \mu)$ to denote explicit dependence of $p_n$ on $\mu$ when necessary. Such a family can be mechanically generated by iterative application of a three-term recurrence relation:
\begin{align}\label{eq:three-term-recurrence}
  x p_n(x) = \sqrt{b_{n}} p_{n-1}(x) + a_n p_n(x) + \sqrt{b_{n+1}} p_{n+1}(x),
\end{align}
where the recurrence coefficients $a_n$ and $b_n$ are functions of the moments of $\mu$. The initial conditions $p_{-1} \equiv 0$ and $p_0 \equiv 1$ are used to seed the recurrence. With $p_n$ defined in this way, the (positive) leading coefficient of $p_n$ has value
\begin{align}\label{eq:leading-coefficient}
  \gamma_n &\coloneqq \prod_{j=0}^n \frac{1}{\sqrt{b_j}}, & p_n(x) &= \gamma_n x^n + \cdots.
\end{align}
The polynomial $p_n$ has $n$ real-valued, distinct roots lying inside the support of $\mu$, and these roots $\left\{x_k\right\}_{k=1}^n$ are nodes for the Gaussian quadrature rule,
\begin{align*}
  \int_{\R} f(x) \dx{\mu}(x) &= \sum_{k=1}^n w_k f(x_k), & f &\in \mathrm{span} \left\{ 1, x, x^2, \ldots, x^{2n-1} \right\},
\end{align*}
where the weights $w_k$ are unique and positive. These nodes and weights can be computed having knowledge only of the recurrence coefficients $a_k$ and $b_k$; numerous modern algorithms accomplish this, with a historically significant procedure given in \cite{golub_calculation_1969}.

\subsection{Induced orthogonal polynomials and measures}
With $p_n$ the orthonormal polynomial family with respect to $\mu$, the collection of polynomials orthogonal with respect to the weighted distribution $p_n^2(x) \dx{\mu(x)}$ with $n$ fixed are called \textit{induced} orthogonal polynomials. We adopt this terminology from \cite{gautschi_set_1993}.

Define the Lebesgue-Stiltjies measure $\mu_n$ and its associated distribution function $F_n$ by
\begin{align}\label{eq:mun-def}
  F_n\left(x; \mu\right) = F_n(x) \coloneqq \int_{-\infty}^x \dx{\mu_n}(t) \coloneqq \int_{-\infty}^x p_n^2(t) \dx{\mu(t)},
\end{align}
with $p_n(\cdot) = p_n(\cdot; \mu)$ the orthonormal family for $\mu$. Note that $\mu_n \ll \mu$, and $F_{n}(\infty) = \mu_n\left(\R\right) = 1$ so that $\mu_n$ is also a probability measure. The measure $\mu_n$ has its own three-term recurrence coefficients $a_{j,n}$ and $b_{j,n}$ for $j=0, \ldots, $ that define a new set of $L^2_{\mu_n}\left(\R\right)$-orthonormal polynomials, which can be generated through the corresponding version of the mechanical procedure \eqref{eq:three-term-recurrence}. One such procedure for generating these coefficients is given in \cite{gautschi_set_1993}.

We will call the measure $\mu_n$ the (order-$n$) \textit{induced} measure for $\mu$, and $F_n$ the corresponding (order-$n$) induced distribution function. %We will call these new recurrence coefficients \textit{induced} recurrence coefficients, following the terminology introduced in CITE.

Our main computational goal is, given $u \in [0,1]$, the evaluation of $F_{n}^{-1}(u)$ for various measures $\mu$. The overall algorithm for accomplishing this is a root-finding method, e.g., bisection or Newton's method. Thus, the goal of finding $F^{-1}_{n}(u)$ also involves the evaluation of $F_{n}(x)$, which is the focus of Section \ref{sec:Fn-eval}. A good root-finding algorithm also requires a reasonable initial guess for the solution. This initial guess is provided by the methodology in Section \ref{sec:ms-interval}.
%Finally, we note that without loss we may assume $n > 0$ since the $n=0$ version of \eqref{eq:mun-def} shows that $F_{0} \equiv F$.

\subsection{Measure modifications}
Our algorithms rely on the ability to compute polynomial measure modifications. That is, given the three-term coefficients $a_n$ and $b_n$ for $\mu$, to compute the coefficients $\widetilde{a}_n$ and $\widetilde{b}_n$ for $\widetilde{\mu}$ defined as
\begin{align*}
  \dx{\widetilde{\mu}}(x) = p(x) \dx{\mu}(x),
\end{align*}
where $p(x)$ is a polynomial, non-negative on the support of $\mu$. This is a well-studied problem \cite{golub_modified_1973,gautschi_interplay_2002,gautschi_orthogonal_2004,narayan_computation_2012}. In particular, one may reduce the problem to iterating over modifications by linear and quadratic polynomials.  We describe in detail how to accomplish linear and quadratic modifications in the appendix, with the particular goal of structuring computations to avoid numerical under- and over-flow when $n$ is large.

%Straightforward implementation of the algorithms described later in this paper results in algorithms that work quite well for small $n$, but that succumb to numerical under- or over-flow for moderate-to-large values of $n$. (Such numerical issues are especially problematic for Freud or half-line Freud weights.) 

%This section states summarizes these relatively delicate computational tasks; a description of how we overcome finite-precision limitations is described in the given Appendix sections.

The following computational tasks described in the Appendix are used to accomplish measure modifications.
\begin{enumerate}
  \item (Appendix \ref{app:auxilliary:ratio}) Evaluation of $r_n$, the ratio of successive polynomials in the orthogonal sequence:
    \begin{align*}
      r_j(x) &\coloneqq \frac{p_j(x)}{p_{j-1}(x)}.
    \end{align*}
    Above, we require $x$ to lie outside the zero set of $p_{j-1}$.
  \item (Appendix \ref{app:auxilliary:quadratic}) Evaluation of a normalized or weighted degree-$n$ polynomial:
    \begin{align*}
      C_n(x) &\coloneqq \frac{p_n(x)}{\sqrt{\sum_{j=0}^{n-1} p_j^2(x)}}, & n &> 0, \; x\in \R
    \end{align*}
    Note that $C_n(x)/r_n(x) \sim 1$ for large enough $|x|$.
  \item (Appendix \ref{app:measure-modifications}) Polynomial measure modifications: given $\mu$ and its associated three-term recurrence coefficients $a_n$ and $b_n$, computation of the three-term recurrence coefficients associated with the measures $\widetilde{\mu}$ and $\widetilde{\widetilde{\mu}}$, defined as 
    \begin{align*}
      \dx{\widetilde{\mu}}(x) &= \pm \left(x - y_0 \right) \dx{\mu}(x), & y_0 &\not\in \mathrm{supp}\,\mu \\
      \dx{\widetilde{\widetilde{\mu}}}(x) &= \left(x - z_0 \right)^2 \dx{\mu}(x), & z_0 &\in \R
    \end{align*}
    The $\pm$ sign in $\widetilde{\mu}$ is chosen so that $\widetilde{\mu}$ is a positive measure.
\end{enumerate}

\begin{figure}
  \includegraphics[width=0.5\textwidth]{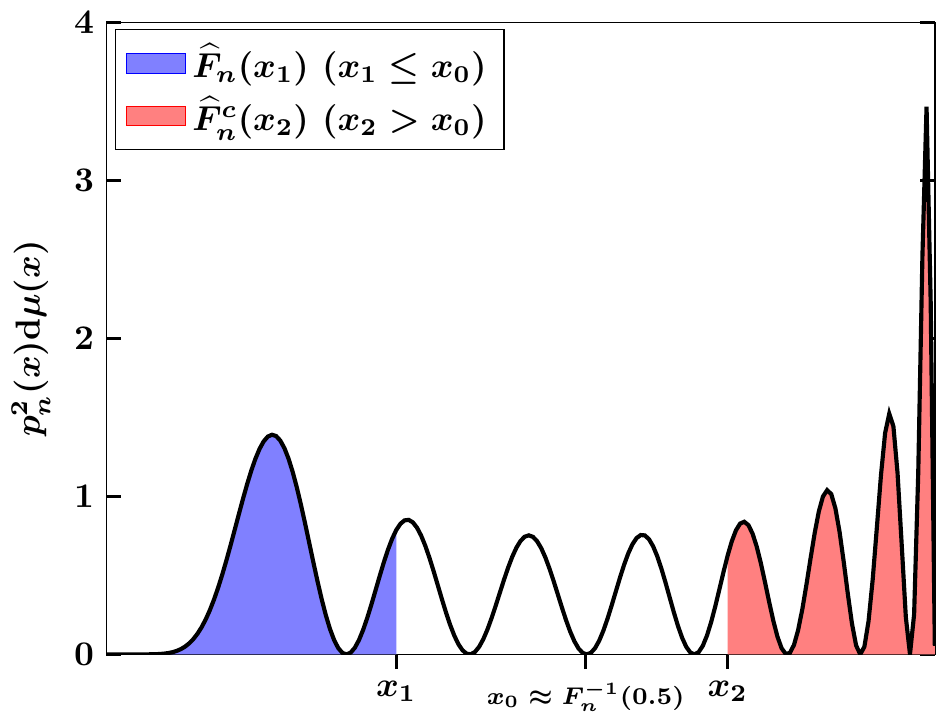}
  \includegraphics[width=0.49\textwidth]{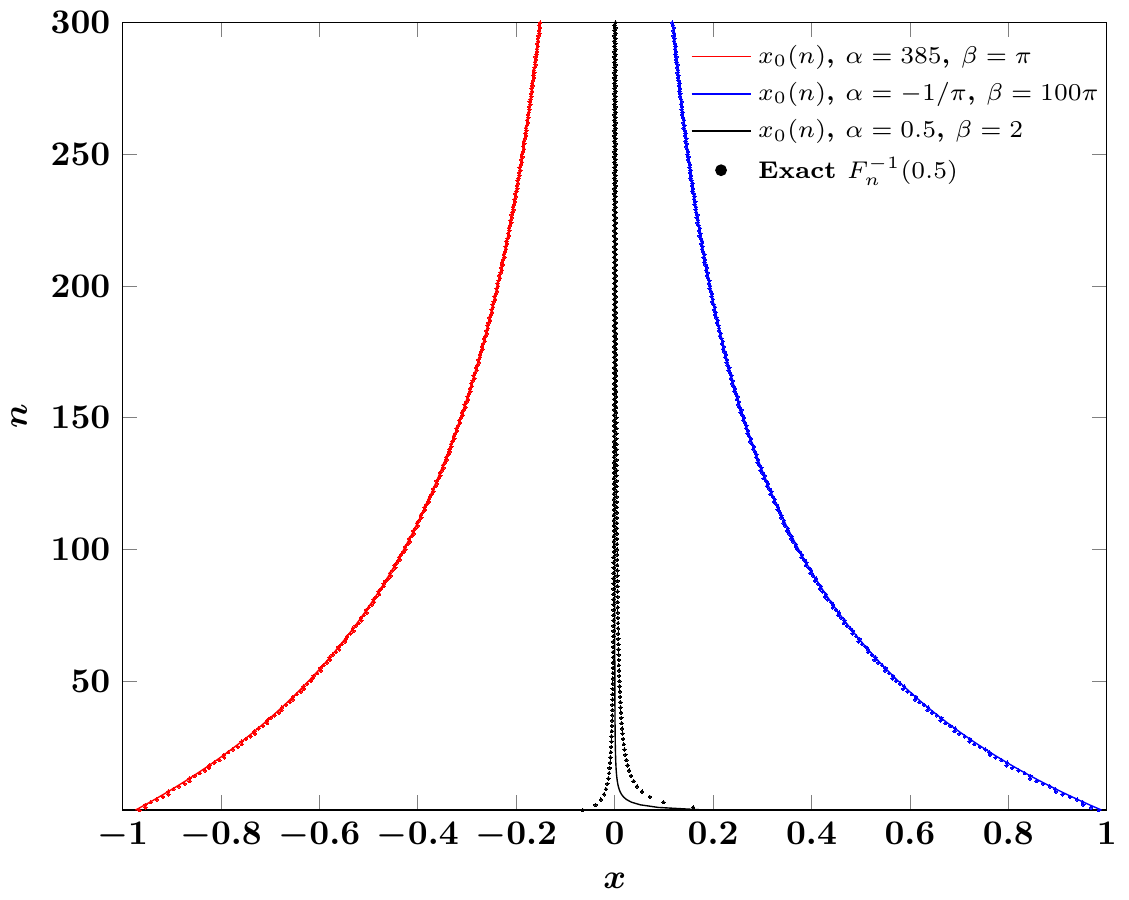}
  \caption{Left: A visual description of the high-level algorithm given by equation \eqref{eq:Fn-computation}. Right: Illustration of exact values of $F_n^{-1}(0.5)$ versus the approximate medians $x_0(n)$ derived in \eqref{eq:x0-jacobi} and \eqref{eq:x0-half-freud} for Jacobi and Half-Freud weights, respectively. Left: Jacobi weights for various $n$ with some choices of $(\alpha, \beta)$. Right: Half-Freud weights for various $n$ with some choices of $(\alpha, \rho)$.}\label{fig:algorithm}
\end{figure}

\section{Evaluation of $F_n$}\label{sec:Fn-eval}
This section develops computational algorithms for the evaluation of the induced distribution $F_n$ defined in \eqref{eq:mun-def}. These algorithms depend fairly heavily on the form of a Lebesgue density $\dx{\mu}(x)$ (i.e., a positive weight function) for the measure $\mu$ on $\R$, but the ideas can be generalized to various measures. We consider the classes of weights enumerated in Table \ref{tab:measures}:
\begin{itemize}
  \item (Jacobi weights) $\dx{\mu}_J(x) = (1-x)^\alpha (1+x)^\beta$ for $x \in [-1,1]$ with parameters $\alpha, \beta > -1$.
  \item (Freud weights) $\dx{\mu}_F(x) = |x|^\rho \exp\left(-|x|^\alpha\right)$ for $x \in \R$ with parameters $\alpha > 0$, $\rho >-1$.
  \item (half-line Freud weights) $\dx{\mu}_{HF}(x) = x^\rho \exp\left(-x^\alpha\right)$ for $x \in [0, \infty)$ with parameters $\alpha > 0$, $\rho >-1$.
\end{itemize}
The induced distribution $F_n$ for Freud weights can actually be written explicitly in terms of the corresponding induced distribution for half-line Freud weights, so most of the algorithm development concentrates on the Jacobi and half-line Freud cases. The strategy for these two latter cases is essentially the same: with $\mu$ in one of the classes above and $n$ fixed, we divide the computation into one of two approximations, depending on the value of $x$. Each approximation is accurate for its corresponding values of $x$. Formally, the algorithm is 
\begin{align}\label{eq:Fn-computation}
  F_n\left(x\right) &= \left\{ \begin{array}{rl} \widehat{F}_n(x), & x \leq x_0\left(\mu, n\right) \\
    1 - \widehat{F}_n^c(x), & x > x_0\left(\mu, n\right)
  \end{array}\right.
\end{align}
where $\widehat{F}_n(x)$ and $\widehat{F}^c_n(x)$ represent computational approximations to $F_n(x)$ and $F^c_n(x)$, respectively, and are the outputs from the algorithms that we will develop. A pictorial description of this is given in Figure \ref{fig:algorithm} (left). The constant $x_0$ is ideally $F_n^{-1}(0.5)$; since we cannot know this value \textit{a priori}, we use potential-theoretic arguments to compute a value $x_0 = x_0\left(\mu, n\right)$ approximating the median of $F_n$,
\begin{align}\label{eq:approximate-median}
  F_n\left(x_0\right) \approx \frac{1}{2}.
\end{align}
For our choice of $x_0$, we can provide no estimates for this approximate equality, but empirical evidence in Figure \ref{fig:algorithm} (right) shows that our choices are very close to the real median, uniformly in $n$. The coming sections concentrate on, for each class of $\mu$ mentioned above, specifying $x_0$ and detailing algorithms for $\widehat{F}_n$ and $\widehat{F}_n^c$.

\subsection{Jacobi weights}
We consider computing induced distribution functions for Jacobi measures $\mu^{(\alpha,\beta)}_J$ as defined in Table \ref{tab:measures}.
%\begin{align*}
%  \dx{\mu_J^{(\alpha,\beta)}}(x) &= \frac{1}{c_J^{(\alpha,\beta)}} (1-x)^\alpha (1+x)^\beta, & x &\in [-1,1],
%\end{align*}
%for fixed parameters $\alpha, \beta > -1$. The constant $c_J^{(\alpha,\beta)}$ is a normalization factor ensuring that $\mu$ is a probability measure, and has value equal to
%\begin{align*}
%  c_J^{(\alpha,\beta)} &= 2^{\alpha+\beta+1} B(\beta+1, \alpha+1), & B(\alpha,\beta) &= \frac{\Gamma(\alpha)\Gamma(\beta)}{\Gamma(\alpha+\beta)} = \int_0^1 t^{\alpha-1} (1-t)^{\beta-1} \dx{t},
%\end{align*}
%where $\Gamma$ is the Euler Gamma function. 
When circumstances are clear, we will write $\mu_J^{(\alpha,\beta)} = \mu$ to avoid notational clutter. We seek the distribution function of the induced measure, 
\begin{align}\label{eq:Pn-def}
  \mu_{J,n}^{(\alpha,\beta)}\left( [-1, x] \right) = \mu_n \left( [-1, x] \right) = F_n(x) &\coloneqq \int_{-1}^x p^2_n(t) \dx{\mu}(t).
\end{align}
%Clearly $\mu_n$ depends on $\alpha$ and $\beta$, but in most cases the parameter values are clear and so in this section we will write $\mu_n$ instead of the more cumbersome $\mu_{J,n}^{(\alpha,\beta)}$, and likewise for $F_n$.

To compute $F_n$, we specify an approximate median $x_0 = x_0\left(\mu, n\right)$ satisfying \eqref{eq:approximate-median}, and construct algorithmic procedures to evaluate $\widehat{F}_n(x) \approx F_n(x)$ (for $x \leq x_0$) and $\widehat{F}_n^c(x) \approx F_n^c(x)$ (for $x > x_0$). Having specified these, we use \eqref{eq:Fn-computation} to compute our approximation to $F_n$.

\subsubsection{Approximating $x_0(n)$}

With $\alpha$, $\beta > -1$ and $n \in \N$ fixed, consider the measure $\mu_J^{\left(\alpha/2n, \beta/2n\right)}$. Note that this measure is still a Jacobi measure since $\frac{\alpha}{2 n} > -1$ and $\frac{\beta}{2 n} > -1$. We may rewrite the integrand in \eqref{eq:Pn-def}
\begin{align*}
  p_n^2(t) \dx{\mu_J^{(\alpha,\beta)}}(t) = \left[ p_n(t) \left(\dx{\mu}_J^{(\alpha/2n, \beta/2n)}(t)\right)^n\right]^2.
\end{align*}
The quantity under the square brackets on the right-hand side is, in the language of potential theory, a weighed polynomial of degree $n$. One result in potential theory characterizes the ``essential" support of this weighted polynomial; in particular, the weighted polynomial decays quickly outside this support. The essential support is an interval, and we take the median $x_0$ of the induced measure $\mu_n$ to be the centroid of this interval.
%is a function whose ``essential" support is $[\theta - \Delta, \theta + \Delta]$. 

%the support of the weighted potential-theoretic equilibrium measure associated to the weight $\dx{\mu^{(\alpha/2n, \beta/2n)}}$ on $[-1,1]$. This support is 
The essential support of the weighted polynomial above is demarcated by the Mhaskar-Rakhmanov-Saff numbers for the asymmetric weight $\dx{\mu_J^{(\alpha/2n, \beta/2n)}}$ on $[-1,1]$. These numbers for this weight are computed explicitly in \cite[pp 206-207]{saff_logarithmic_1997}. When $\alpha, \beta \geq 0$, this support interval is $\left[ \theta - \Delta, \theta + \Delta \right]$, with
\begin{align*}
  \theta &= \frac{ \beta^2 - \alpha^2}{\left(2 n + \alpha + \beta\right)^2}, & 
  \Delta &= \frac{4 \sqrt{n(n + \alpha + \beta)(n + \alpha) (n + \beta)}}{(2 n + \alpha + \beta)^2}
\end{align*}
%The utility of this interval is that polynomials of degree $n$ weighted by $\left(\dx{\mu^{(\alpha/2n,\beta/2n)}}(x)\right)^n$ achieve their supremum on this interval and decay quickly to zero outside of this interval. This motivates the heuristic that the integrand in \eqref{eq:Pn-def}
%\begin{align*}
%  p_n^2(t) \dx{\mu^{(\alpha,\beta)}}(t) = \left[ p_n(t) \left(\dx{\mu}^{(\alpha/2n, \beta/2n)}(t)\right)^n\right]^2
%\end{align*}
%is a function whose ``essential" support is $[\theta - \Delta, \theta + \Delta]$. 
Since this is the interval where most of the ``mass" of the integral in \eqref{eq:Pn-def} lies, we set $x_0$ to be the centroid $\theta$ of this interval:
\begin{align}\label{eq:x0-jacobi}
  x_0\left(\mu^{(\alpha,\beta)},n\right) &= \frac{\beta^2 - \alpha^2}{\left(2 n + \alpha + \beta\right)^2}, & \alpha, \beta > -1, \;\;& n > 0.
\end{align}
Note that the definitions of $\theta, \Delta$ require $\alpha$ and $\beta$ to be non-negative. Without mathematical justification, we extend the formula \eqref{eq:x0-jacobi} to valid negative values of $\alpha, \beta$ as well. Figure \ref{fig:algorithm} compares $x_0$ and $F_n^{-1}\left(0.5\right)$ for certain choices of $\alpha$ and $\beta$.

\subsubsection{Computing $F_n(x)$}
First assume that $x \leq x_0$, with $x_0$ defined in \eqref{eq:x0-jacobi}, and define $A \in \N_0$ as 
\begin{align*}
  A &\coloneqq \left\lfloor |\alpha| \right\rfloor, & \alpha - A &\in (-1,1),
\end{align*}
where $\lfloor\cdot\rfloor$ is the floor function. We transform the integral \eqref{eq:Pn-def} over $[-1,x]$ onto the standard interval $[-1,1]$ via the substitution $u = \frac{2}{x+1} (t+1) - 1$:
\begin{align}\nonumber
  F_n(x) &= \frac{1}{c_J^{(\alpha,\beta)}} \int_{-1}^x p_n^2\left(t\right) (1 - t)^\alpha (1 + t)^\beta \dx{t} \\\nonumber
         &= \left(\frac{x+1}{2}\right)^{\beta+1} \frac{c_J^{(0,\beta)}}{c_J^{(\alpha,\beta)}} \int_{-1}^1 \left(2 - \frac{1}{2}(u+1)(x+1)\right)^{\alpha-A} U_{2n+A}(u) \dx{\mu^{(0,\beta)}(u),}
%         &= \left(\frac{x+1}{2}\right)^{\beta+1} \frac{1}{c_J^{(\alpha,\beta)}} \int_{-1}^1 p_n^2\left( \frac{1}{2}(u+1)(x+1) - 1\right) (1+u)^\beta \left(2 - \frac{1}{2}(u+1)(x+1)\right)^\alpha \dx{u} \\\label{eq:Fn-temp-jacobi}
%         &= \left(\frac{x+1}{2}\right)^{\beta+1} \frac{c_J^{(0,\beta)}}{c_J^{(\alpha,\beta)}} \int_{-1}^1 p_n^2\left( \frac{1}{2}(u+1)(x+1) - 1\right) \left(2 - \frac{1}{2}(u+1)(x+1)\right)^\alpha \dx{\mu^{(0,\beta)}(u)}
\end{align}
where $U_{2n+A}(u)$ is a degree-$(2n+A)$ polynomial given by 
\begin{align*}
  U_{2n+A}(u) &= p_n^2\left( \frac{1}{2}(u+1)(x+1) - 1\right) \left(2 - \frac{1}{2}(u+1)(x+1)\right)^A \\
              &= \gamma^2_n \left(\frac{x+1}{2}\right)^{2n+A} \underbrace{\prod_{k=1}^A \left[\left( \frac{3-x}{1+x}\right) - u \right]}_{(a)} \underbrace{\prod_{j=1}^n \left(u - \left( \frac{2}{x+1} \left( x_{j,n} + 1 \right) - 1\right) \right)^2}_{(b)},
\end{align*}
where $\left\{x_{j,n}\right\}_{j=1}^n$ are the $n$ zeros of $p_n(\cdot)$. Since we explicitly know the polynomial roots of $U_{2n+A}$, we can absorb the term marked $(a)$ into the measure $\dx{\mu}^{(0,\beta)}(u)$ via $A$ linear modifications, and we can likewise absorb the term $(b)$ via $n$ quadratic modifications. (See Appendix \ref{app:measure-modifications}.) Thus, define
\begin{align}\label{eq:modified-mu-def-jacobi}
  \dx{\widetilde{\mu}_n}(u) = U_{2n+A}(u) \dx{\mu}^{(0,\beta)}(u),
\end{align}
which is a modified measure whose recurrence coefficients can be computed via successive application of the linear and quadratic modification methods in Appendix \ref{app:measure-modifications}. Thus,
\begin{align}\label{eq:Fn-transformed}
  F_n(x) = \left(\frac{x+1}{2}\right)^{\beta+1} \frac{c_J^{(0,\beta)}}{c_J^{(\alpha,\beta)}} \int_{-1}^1 \left(2 - \frac{1}{2}(u+1)(x+1)\right)^{\alpha-A} \dx{\widetilde{\mu}_n}(u).
\end{align}
The integrand above has a root ($\alpha > 0$) or singularity ($\alpha < 0$) at $u = \frac{3-x}{x+1} = 1 + 2 \frac{1-x}{1+x} \geq 1 + 2 \frac{1-x_0}{1+x_0}$; this root is far outside the interval $[-1,1]$ unless $\beta$ is very large and both $n$ and $\alpha$ are small. The integrand is therefore a positive, monotonic, smooth function on $[-1,1]$, taking values between $1-x$ and $2$; we use an order-$M$ $\widetilde{\mu}_n$-Gaussian quadrature to efficiently evaluate it. With $\left(\widetilde{u}_m, \widetilde{w}_m\right)_{m=1}^M$ denoting the nodes and weights, respectively, of this quadrature rule, we compute
\begin{align}\label{eq:I-def}
  %I(x) &\coloneqq \int_{-1}^1 \left(2 - \frac{1}{2}(u+1)(x+1)\right)^a \dx{\widetilde{\mu}_n}(x), \\
  I(x) &\coloneqq \sum_{m=1}^M \widetilde{w}_m \left(2 - \frac{1}{2}\left(\widetilde{u}_m+1\right)(x+1)\right)^{\alpha-A}, \\
  \label{eq:Fn-I}
  \widehat{F}_n(x) &= \left(\frac{x+1}{2}\right)^{\beta+1} \frac{I(x)}{2^a (\beta + 1) B(\beta+1, \alpha+1)}
\end{align}
The entire procedure is summarized in Algorithm \ref{alg:Fn-Jacobi}.

In order to compute $\widehat{F}^c_n(x)$ for $x > x_0$, we use symmetry. Since $x \gets -x$ interchanges the parameters $\alpha$ and $\beta$, then
\begin{align*}
  \widehat{F}_n^{(\alpha,\beta),c}\left(x\right) = \int_x^1 \left[ p_n^{(\alpha,\beta)}(t) \right]^2 \dx{\mu^{(\alpha,\beta)}}(t) = \int_{-1}^{-x} \left[ p_n^{(\beta,\alpha)}(t)\right]^2 \dx{\mu}^{(\beta,\alpha)}(t) = \widehat{F}_n^{(\beta,\alpha)}(-x).
\end{align*}
Note that if $x > x_0\left(\mu^{(\alpha, \beta)}, n\right)$, then $-x < x_0 \left(\mu^{(\beta, \alpha)}, n\right)$. Thus, $\widehat{F}_n^c$ can be computed via the same algorithm for $\widehat{F}_n$, but with different values for $\alpha$, $\beta$, and $x$. This is also shown in Algorithm \ref{alg:Fn-Jacobi}.

%When $x > 0$, we can use symmetry to evaluate this integral: With $\mu^{(a,b)}$ the Jacobi (probability) measure with associated order-$n$ induced distribution $F_n^{(a,b)}$, we have:
%\begin{align*}
%  F_n^{(a,b)}(x) &= \int_{-1}^x \left[ p_n^{(a,b)}(t) \right]^2 \dx{\mu^{(a,b)}(t)} = 1 - \int_{x}^1 \left[ p_n^{(a,b)}(t) \right]^2 \dx{\mu^{(a,b)}(t)} \\
%                 &= 1 - \int_{-1}^{-x} \left[ p_n^{(b,a)}(t) \right]^2 \dx{\mu^{(b,a)}(t)} = 1 - F_n^{(b,a)}(-x)
%\end{align*}
%We therefore use the following procedure for evaluation of $F^{(a,b)}(x)$:
%\begin{align*}
%  F_n^{(a,b)}(x) = \left\{\begin{array}{rl} \textrm{Algorithm \ref{alg:Fn-Jacobi}}(a, b, n, x), & x \leq 0, \\
%1 - \textrm{Algorithm \ref{alg:Fn-Jacobi}}(b, a, n, -x), & x > 0. \end{array}\right.
%\end{align*}

\begin{algorithm2e}[H]%need H when used with beamer
\SetKwInOut{Input}{input}\SetKwInOut{Output}{output}
\Input{$\alpha, \beta > -1$: Jacobi polynomial parameters}
\Input{$n\in \N_0$ and $x \in [-1,1]$: Order of induced polynomial and measure $\mu_n$ and value $x$.}
\Input{$M \in \N$: Quadrature order for approximate computation of $F_n(x)$.}
\Output{$\widehat{F}_n\left(x; \mu^{(\alpha, \beta)}\right)$}
\BlankLine
If $x > x_0\left(\mu{(\alpha,\beta)}, n\right)$, return $1 - \widehat{F}_n(-x; \mu^{(\beta, \alpha)})$\;
Compute $n$ zeros, $\left\{x_{j,n}\right\}_{j=1}^n$ of $p_n = p_n\left(\cdot; \mu^{(\alpha,\beta)}\right)$, and leading coefficient $\gamma_n$ of $p_n$.\;
Compute recurrence coefficients $a_j$ and $b_j$ associated to $\mu^{(0,\beta)}$ for $0 \leq j \leq M + A + 2 n$.\;
\For{$j=1, \ldots, n$}{
  Quadratic measure modification \eqref{eq:quadratic-modification}: update $a_n$, $b_n$ for $n=0, \ldots, M + A + 2(n-j)$ with modification factor $\left(u - \left( \frac{2}{x+1} \left(x_{j,n} + 1\right) - 1\right)\right)^2$.\\
  %Update $a_n$, $b_n$ for $n=0, \ldots, M + A + 2(n-j)$: Use \eqref{eq:quadratic-modification} to modify measure by factor $\left(u - \left( \frac{2}{x+1} \left(x_{j,n} + 1\right) - 1\right)\right)^2$.\\
  %Scale $b_0 \gets b_0 \left(\frac{x+1}{2}\right)^2 \exp\left(\frac{1}{n} \log \gamma_n^2 \right)$.
  %Scale $b_0 \gets b_0 \left(\frac{x+1}{2}\right)^2 \gamma_n^{2/n}$.
}
\For{$k=1, \ldots, A$}{
  Use linear modification \eqref{eq:linear-modification} to update $a_n$, $b_n$ for $n=0, \ldots, M + (A-k)$ with modification factor $\left(u - \left(\frac{3-x}{1+x}\right)\right)$.\\
  %Scale $b_0 \gets b_0 \left(\frac{x+1}{2}\right)$.
}
$b_0 \gets b_0 \left(\frac{x+1}{2}\right)^{2n + A} \gamma_n^2$\\
Compute $M$-point Gauss quadrature $\left(\widetilde{u}_m, \widetilde{w}_m\right)_{m=1}^M$ associated with measure $\widetilde{\mu}_n$ via $\left\{\left(a_j, b_j\right)\right\}_{j=0}^M$.\\
Compute the integral $I$ in \eqref{eq:I-def}, and return $\widehat{F}_n\left(x; \mu^{(\alpha,\beta)}\right)$ given by \eqref{eq:Fn-I}.
\caption{Computation of $\widehat{F}_n(x)$, approximating $F_n(x)$ for $\mu$ corresponding to a Jacobi polynomial measure.}\label{alg:Fn-Jacobi}
\end{algorithm2e}

\begin{theorem}\label{thm:jacobi-Fn-estimate}
  With $\mu^{(\alpha,\beta)}$, $n \in \N$, and $x \in [-1,1]$ all given, assume that $x \leq x_0$ with $x_0$ as in \eqref{eq:x0-jacobi}. Then the output $\widehat{F}_n(x)$ from Algorithm \ref{alg:Fn-Jacobi} using an $M$-point quadrature rule satisfies
  \begin{align}\label{eq:jacobi-Fn-estimate}
    \left| F_n(x) - \widehat{F}_n(x) \right| \leq C(\alpha,\beta,n,M) \prod_{j=0}^M b_j\left(\widetilde{\mu}_n\right),
  \end{align}
  where $b_j\left(\widetilde{\mu}_n\right)$ are the $b_j$ three-term recurrence coefficients associated to the $x$-dependent measure $\widetilde{\mu}_n$ defined in \eqref{eq:modified-mu-def-jacobi}. The constant $C$ is
  \begin{align*}
    %C(\alpha,\beta,n,M) = \frac{2^{\beta+1}}{(\beta+1) B(\beta+1,\alpha+1)} \left( \frac{x_0(n) + 1}{4} \right)^{2 M+\beta} \binom{\alpha}{2 M}
    C(\alpha,\beta,n,M) = \frac{2^{\beta+1-A}}{(\beta+1) B(\beta+1,\alpha+1)} \left( \frac{x_0(n) + 1}{4} \right)^{2 M+\beta+1} %\binom{\alpha}{2 M}
  \end{align*}
\end{theorem}
Note that $C(\alpha,\beta,n,M)$ on the right-hand side of \eqref{eq:jacobi-Fn-estimate} is explicitly computable once $\alpha$, $\beta$, $M$, and $n$ are fixed, and only the product involving $b_j\left(\widetilde{\mu}_n\right)$ depends on $x$. Furthermore, this last quantity is explicitly computed in Algorithm \ref{alg:Fn-Jacobi}, so that a rigorous error estimate for the algorithm can be computed before its termination.

Since $\frac{x_0+1}{4} < \frac{1}{2}$, then the estimate \eqref{eq:jacobi-Fn-estimate} also hints at exponential convergence of the quadrature rule strategy, assuming that the factors $b_j\left(\widetilde{\mu}_n\right)$ can be bounded or controlled. We cannot provide these bounds, although we do know the asymptotic behavior $b_j\left(\widetilde{\mu}_n\right) \rightarrow \frac{1}{4}$ as $j \rightarrow \infty$ \cite[Remark 3.1.10]{nevai_orthogonal_1980}. Also, since $x_0 \in [-1,1]$ for any $n > 0$ then, uniformly in $n > 0$, $C(\alpha, \beta, n, M) \leq C'(\alpha, \beta) 4^{-M}$.  Thus, for a fixed $x$ we expect that the estimate in Theorem \ref{thm:jacobi-Fn-estimate} behaves like
\begin{align*}
  \left| F_n(x) - \widehat{F}_n(x) \right| \leq C(\alpha,\beta,n,M) \prod_{j=0}^M b_j\left(\widetilde{\mu}_n\right) \lesssim C''(\alpha, \beta, x) 4^{-2M},
\end{align*}
showing exponential convergence with respect to $M$. However, we cannot prove this latter statement.
%Note that the bound above is explicitly computable before termination of Algorithm \ref{alg:Fn-Jacobi}: $C$ is computable at the outset, and $\prod_{j=0}^M b_j(\widetilde{\mu}_n)$ is computable after the measure modification portion of the algorithms. 
\begin{proof}
  The result is a relatively straightforward application of known error estimates for Gaussian quadrature with respect to non-classical weights. We use the notation of Algorithm \ref{alg:Fn-Jacobi}: $\widetilde{u}_m$ and $\widetilde{w}_m$ denote the $M$-point $\widetilde{\mu}_n$-Gaussian quadrature nodes and weights, respectively. We start with the Corollary to Theorem 1.48 in \cite{gautschi_orthogonal_2004}, stating that if $f(\cdot)$ is infinitely differentiable on $[-1,1]$, then 
  \begin{align*}
    \left| \int_{-1}^1 f(u) \dx{\widetilde{\mu}_n(u)} - \sum_{j=1}^M \widetilde{w}_m f\left(\widetilde{u}_m\right) \right| = \frac{f^{(2M)}(\tau)}{(2 M)!} \int_{-1}^1 \prod_{j=1}^M \left( u - \widetilde{u}_m\right)^2 \dx{\widetilde{\mu}_n}(u)
  \end{align*}
  for some $\tau \in (-1,1)$. Noting that since $\widetilde{u}_m$ are the zeros of the degree-$M$ $\widetilde{\mu}_n$-orthogonal polynomial, then 
  \begin{align*}
    \int_{-1}^1 \prod_{j=1}^M \left( u - \widetilde{u}_m\right)^2 \dx{\widetilde{\mu}_n}(u) = \prod_{j=0}^M b_j\left(\widetilde{\mu}_n\right) \int_{-1}^1 p^2_M\left(u; \widetilde{\mu}_n\right) \dx{\widetilde{\mu}_n}(u) = \prod_{j=0}^M b_j \left(\widetilde{\mu}_n\right).
  \end{align*}
  From \eqref{eq:Fn-transformed}, the integral we wish to approximate has integrand $f(u) = \left(2 - \frac{1}{2} (u+1)(x+1)\right)^{\alpha-A}$. Then for any $\tau \in (-1,1)$ and any $x \leq x_0$,
  \begin{align*}
    \left| f^{(2M)}(\tau) \right| &= \left(\frac{x+1}{2}\right)^{2 M} \left(2 - \frac{1}{2} (x+1)(\tau+1) \right)^{\alpha - A - 2 M} \prod_{j=0}^{2M-1} |\alpha - A - j| \\
                                  &\leq \left(\frac{x_0+1}{2}\right)^{2 M} 2^{\alpha - A - 2 M} \prod_{j=0}^{2M-1} |\alpha - A - j| = 2^{\alpha-A} \left(\frac{x_0+1}{4} \right)^{2M} \prod_{j=0}^{2M-1} |\alpha - A - j|
  \end{align*}
  Then
  \begin{align*}
    \left| F_n(x) - \widehat{F}_n(x) \right| &= \left(\frac{x+1}{2}\right)^{\beta+1} \frac{c_J^{(0,\beta)}}{c_J^{(\alpha,\beta)}} \left| \int_{-1}^1 f(u) \dx{\widetilde{\mu}_n(u)} - \sum_{j=1}^M \widetilde{w}_m f\left(\widetilde{u}_m\right) \right| \\
                                             &\leq \left(\frac{x+1}{2}\right)^{\beta+1} \frac{c_J^{(0,\beta)}}{c_J^{(\alpha,\beta)}} 2^{\alpha-A} \left(\frac{x_0+1}{4} \right)^{2M} \frac{\prod_{j=0}^{2M-1} |\alpha - A - j|}{(2 M)!} \prod_{j=0}^M b_j \left(\widetilde{\mu}_n\right) \\
                                             %&= 2^{\alpha + \beta + 1} \frac{c_J^{(0,\beta)}}{c_J^{(\alpha,\beta)}} \left(\frac{x_0 +1}{4} \right)^{2 M + \beta + 1} \binom{\alpha}{2M} \prod_{j=0}^M b_j\left(\widetilde{\mu}_n\right)
  \end{align*}
  Since $\alpha - A \in (-1,1)$, then 
  \begin{align*}
    \left| F_n(x) - \widehat{F}_n(x) \right| &\leq 2^{\alpha + \beta + 1-A} \frac{c_J^{(0,\beta)}}{c_J^{(\alpha,\beta)}} \left(\frac{x_0+1}{4} \right)^{2M+\beta+1} \prod_{j=0}^M b_j \left(\widetilde{\mu}_n\right) 
  \end{align*}
  The result follows by direct computation of $2^{\alpha + \beta + 1-A} \frac{c_J^{(0,\beta)}}{c_J^{(\alpha,\beta)}}$.
\end{proof}
We verify spectral convergence of the scheme empirically in Figure \ref{fig:jacobi-error}, which also illustrates that for the test cases shown, one could choose $M$ to bound errors uniformly in $x$. The figure also shows that qualitative error behavior is uniform even for extremely large values of $n$ and/or $\alpha$ or $\beta$. Based on these results, taking $M=10$ appears sufficient uniformly over all $(\alpha, \beta, n)$.\footnote{Not shown: we have verified this for numerous values of $(\alpha, \beta, n)$.} Finally, extending the estimate in \eqref{eq:jacobi-Fn-estimate} to the case $x > x_0$ can be accomplished by permuting $\alpha$ and $\beta$ as is done for that case in Algorithm \ref{alg:Fn-Jacobi}. 
\begin{figure}
  \begin{center}
    \includegraphics[width=0.49\textwidth]{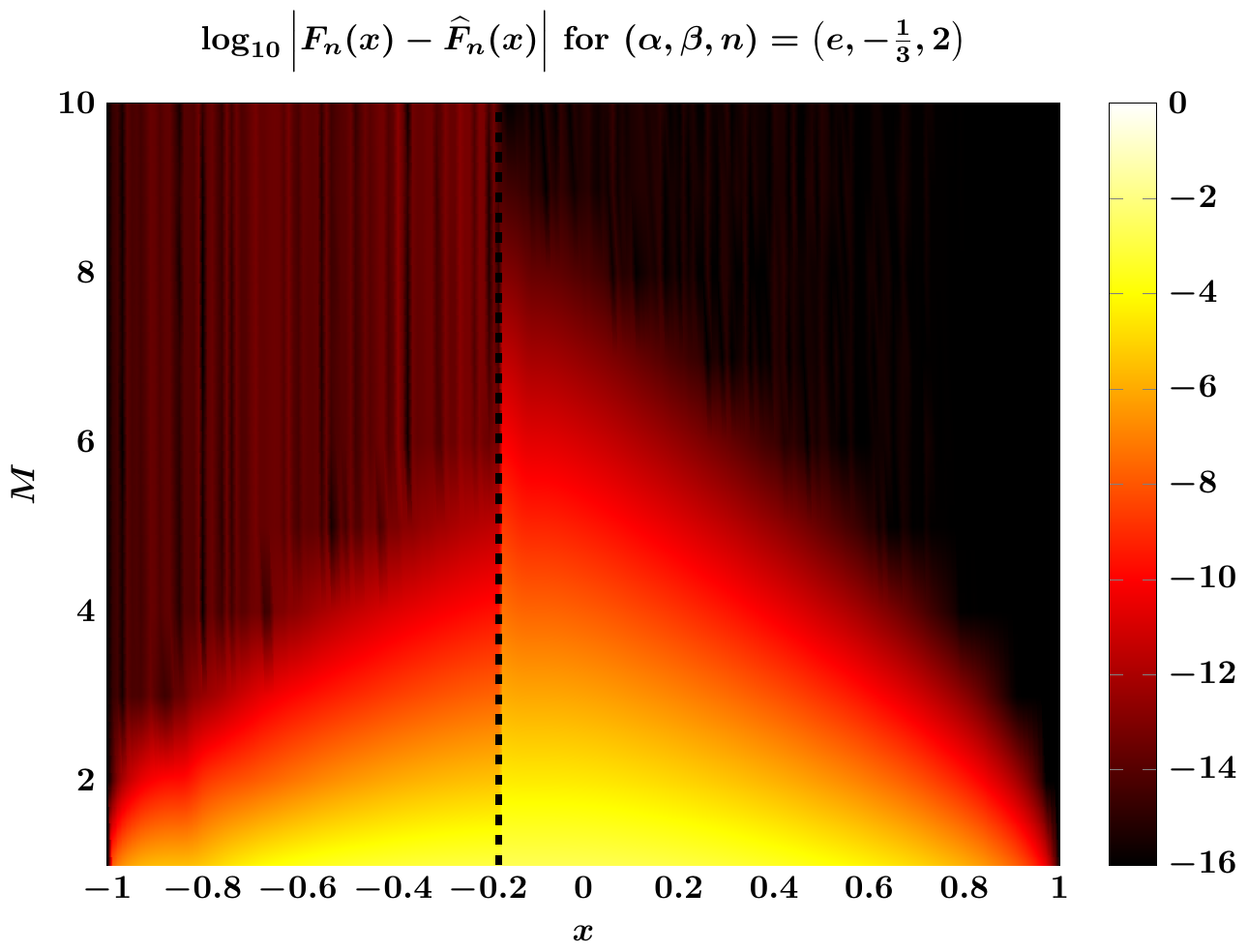}
    \includegraphics[width=0.49\textwidth]{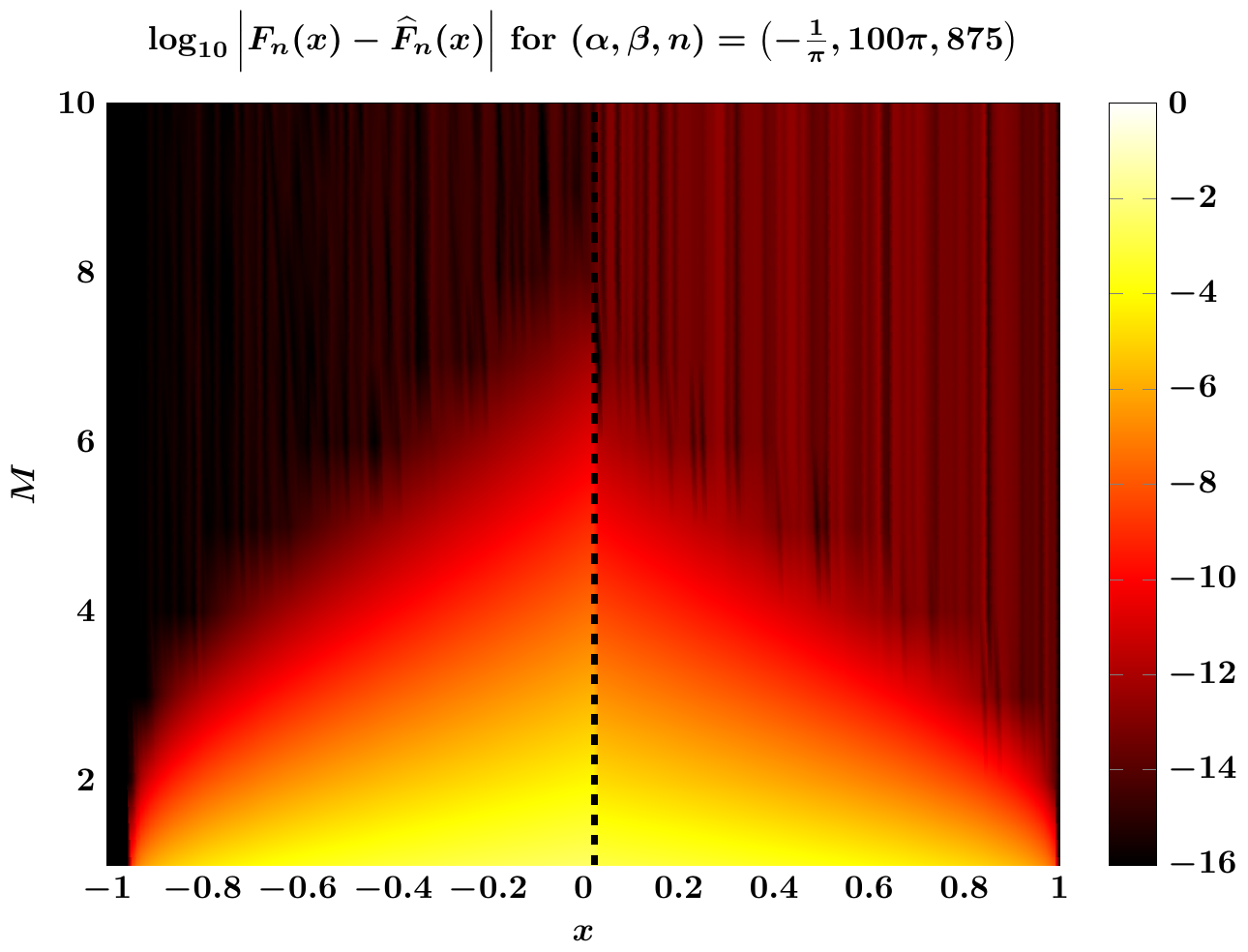}
  \end{center}
  \caption{Color plot of $\log_{10} \left| F_n(x) - \widehat{F}_n(x) \right|$ for Jacobi measures with certain choices of the order $n$ and various quadrature size $M$. Left: $(\alpha, \beta, n) = \left(e, -\frac{1}{3}, 2\right)$. Right: $(\alpha, \beta, n) = \left( -\frac{1}{\pi}, 100\pi, 875\right)$. In both figures the vertical dashed black line indicates $x_0(n)$. For these test cases, we obtain more than 10 digits of accuracy with only $M=10$, uniformly in $x$.}\label{fig:jacobi-error}
\end{figure}

\subsection{Half-line Freud weights}
In this section we consider the half-line Freud measure $\mu^{(\alpha,\rho)}_{HF}$ as defined in Table \ref{tab:measures}. These algorithms require the recurrence coefficients for $\mu_{H F}$; these coefficients are in general not easy to compute when $\alpha \neq 1$. We show in Appendix \ref{app:recurrence-coefficients} that these recurrence coefficients can be determined from the recurrence coefficients for Freud weights, but recurrence coefficients for Freud weights are themselves relatively difficult to tabulate \cite{magnus_freuds_1999,assche_discrete_2005}. In our computations we use the methodology of \cite{gautschi_variable-precision_2009} to compute Freud weight recurrence coefficients (and hence half-line Freud coefficients). We note that the methodology of \cite{gautschi_variable-precision_2009} is computationally onerous: For a fixed $\alpha$ and $\rho$, it required a day-long computation to obtain 500 recurrence coefficients.

%  \item (Appendix \ref{app:recurrence-coefficients}) Computation of recurrence coefficients for Freud and generalized Freud weights. We describe a procedure that uses only Freud weight recurrence coefficients to compute induced polynomial distribution functions for \textit{both} the Freud and half-Freud cases. The procedure we use to compute Freud weight recurrence coefficients is exactly that proposed in \cite{gautschi_set_1993}. We present one theoretical statement that provides the connection between Freud and half-Freud weights. This theoretical result is standard and well-known to analysts, but to our knowledge appears to have not yet been exploited in the numerical algorithms setting.

%Let $\mu_{HF}$ be a measure on $[0, \infty)$ with density
%\begin{align}\label{eq:generalized-freud-weight}
%  \dx{\mu_{HF}^{(\alpha,\rho)}}(x) &= \frac{1}{c_{HF}^{\left(\alpha,\rho\right)}} x^{\rho} \exp\left(-x^\alpha\right), & x &\geq 0,
%\end{align}
%with parameters $\rho > -1$ and $\alpha > \frac{1}{2}$, and 
%\begin{align}\label{eq:generalized-freud-normalization}
%  c_{HF}^{(\alpha,\rho)} = \frac{1}{\alpha} \Gamma\left(\frac{\rho+1}{\alpha}\right)
%\end{align}
Again in this subsection we write $\mu^{(\alpha,\rho)}_{HF} = \mu$ and similarly for $\mu_n$, $F$, $F_n$, etc. We restore these super- and subscripts when ambiguity arises without them.

We accomplish computation of $F_n$ for this measure with largely the same procedure as for Jacobi measures. Like in the Jacobi case, the details of the procedure we use differ depending on whether $x$ is closer to the left-hand end of $\supp \mu$ (which is $x=0$ here), or to the right-hand end of the $\supp \mu$ (which is $x=\infty$). We determine this delineation again by means of potential theory. 

\subsubsection{Computation of $x_0$}
As with the Jacobi case, we take $x_0$ to be the midpoint of the ``essential" support for $p_n^2(x) \dx{\mu}_{HF}^{(\alpha,\rho)}$, the latter of which is approximately the support of the weighted equilibrium measure associated to $\left[\dx{\mu}_{HF}^{\left(\alpha, \rho\right)}\right]^{1/2n}$. However, directly computing the support of this equilibrium measure is difficult. Thus, we resort to a more \textit{ad hoc} approach.

To derive our approach, we first compute the exact support of special cases of our Half-line Freud measures. 
\begin{itemize}
    \item The support of the weighted equilibrium measure associated to the measure
  \begin{align}\label{eq:alph-1-special-case}
    \dx{\mu}(x) &= x^s \exp(-\lambda x), & x \in [0, \infty), & s\geq 0 &\; \lambda > 0
  \end{align}
  is the interval $[\theta - \Delta, \theta + \Delta]$, with these values given by \cite{saff_logarithmic_1997} 
  \begin{align*}
    \theta &= \frac{s+1}{\lambda}, & \Delta &= \frac{\sqrt{2s+1}}{\lambda}
  \end{align*}
  This interval is the ``essential" support for any function of the form $p_n(x) \left(\dx{\mu}(x)\right)^n$ where $p_n$ is a degree-$n$ polynomial.
  \item The second special case is for arbitrary $\alpha$, but $\rho = 0$. The support of the weighted equilibrium measure for $\sqrt{\dx{\mu}_{HF}^{(\alpha,0)}}$ in this case is the interval $[0, k_n(\alpha)]$, where 
  \begin{align*}
    k_n\left(\alpha\right) = k_n \coloneqq n^{1/\alpha} \left(\frac{2 \sqrt{\pi} \Gamma(\alpha)}{\Gamma\left(\alpha + \frac{1}{2}\right)}\right)^{1/\alpha}. 
  \end{align*}
  are the Mhaskar-Rakhmanov-Saff numbers for $\sqrt{\mu_{HF}^{(\alpha,0)}}$ \cite{mhaskar_where_1985}. 
\end{itemize}

We now derive our approximation for the case of general $\alpha$, $\rho$. To approximate where $p_n \sqrt{ \dx{\mu}_{HF}^{(\alpha,\rho)}}$ is supported, consider 
\begin{align*}
  p_n(x) \left( \dx{\mu}_{HF}^{\left(\alpha, \rho\right)}(x) \right)^{1/2} &\propto p_n(x) \left[ x^{\rho/2n} \exp\left(-\frac{1}{2n} x^\alpha \right) \right]^{n} \\
                                                                        &\stackrel{u = x^\alpha}{=} p\left(u^{1/\alpha}\right) \left[ u^{\rho/2 n \alpha} \exp\left(-\frac{1}{2n} u\right) \right]^{n} \\
                                                                        &= q_{n/\alpha}(u) \left[ u^{\rho/2n} \exp\left(-\frac{\alpha}{2 n} u \right) \right]^{n/\alpha},
\end{align*}
where we have introduced $q_{n/\alpha}$, which is a ``polynomial" of ``degree" $n/\alpha$\footnote{More formally, it is a potential with ``mass" $n/\alpha$.}. Note that in the variable $u$, the weight function under square brackets in the last expression is of the form \eqref{eq:alph-1-special-case}. Concepts in potential theory extend to generalized notions of polynomial degree, and so we may apply our formulas for $\theta$ and $\Delta$ with $s = \frac{\rho}{2n}$ and $\lambda = \frac{\alpha}{2 n}$. These formulas imply that the ``essential" support for the $u$ variable is
\begin{align*}
  \theta - \Delta \leq u = x^{\alpha} \leq \theta + \Delta
\end{align*}
Therefore, to obtain appropriate limits on the variable $x$, we raise the endpoints $\theta \pm \Delta$ to the $1/\alpha$ power. However, we now require a correction factor. To see why, we compute the right-hand side of our computed support interval when $\rho = 0$
\begin{align}\label{eq:temp-support-rhs}
  \left(\theta + \Delta\right)^{1/\alpha} = \left( 2n + 2 \sqrt{ n^2} \right)^{1/\alpha} = 2^{2/\alpha} n^{1/\alpha},
\end{align}
and compare this with the exact value $k_n(\alpha)$ computed above. We note that while $k_n(\alpha) \sim n^{1/\alpha}$ matches the $n$-behavior of \eqref{eq:temp-support-rhs}, the constant is wrong. We thus multiply the endpoints $\left( \theta \pm \Delta\right)^{1/\alpha}$ by the appropriate constant to match the $\rho = 0$ behavior of $k_n(\alpha)$. The net result then, for arbitrary $\alpha, \rho$, is the approximation
\begin{subequations}\label{eq:half-freud-interval}
  \begin{align}\label{eq:half-freud-amin}
    a_{\pm}\left(n, \alpha, \rho\right) &= \left( \frac{\sqrt{\pi} \Gamma(\alpha)}{2 \Gamma\left(\alpha + \frac{1}{2}\right)}\right)^{1/\alpha} \left( \rho + 2 n \pm 2 \sqrt{n^2 + n \rho} \right)^{1/\alpha} \\\label{eq:x0-half-freud}
    x_0\left(n;\mu_{HF}^{(\alpha,\rho)}\right) &= \frac{1}{2} \left( a_{-}\left(n, \alpha, \rho\right) + a_{+}\left(n, \alpha, \rho\right)\right)
  %x_0\left(n;\mu_{HF}^{(\alpha,\rho)}\right) = \frac{1}{2} \left( \frac{\sqrt{\pi} \Gamma(\alpha)}{2 \Gamma\left(\alpha + \frac{1}{2}\right)}\right)^{1/\alpha} \left[ \left( \rho + 2 n - 2 \sqrt{n^2 + n \rho} \right)^{1/\alpha} + \left( \rho + 2 n + 2 \sqrt{n^2 + n \rho} \right)^{1/\alpha} \right].
\end{align}
\end{subequations}
Figure \ref{fig:half-freud-mrs} compares the intervals demarcated by $a_{-}$ and $a_{+}$ versus $F_n^{-1}\left([0.01, 0.99]\right)$, the latter of which contains ``most" of the support of $F_n$.
\begin{figure}
  \begin{center}
    \includegraphics[width=0.49\textwidth]{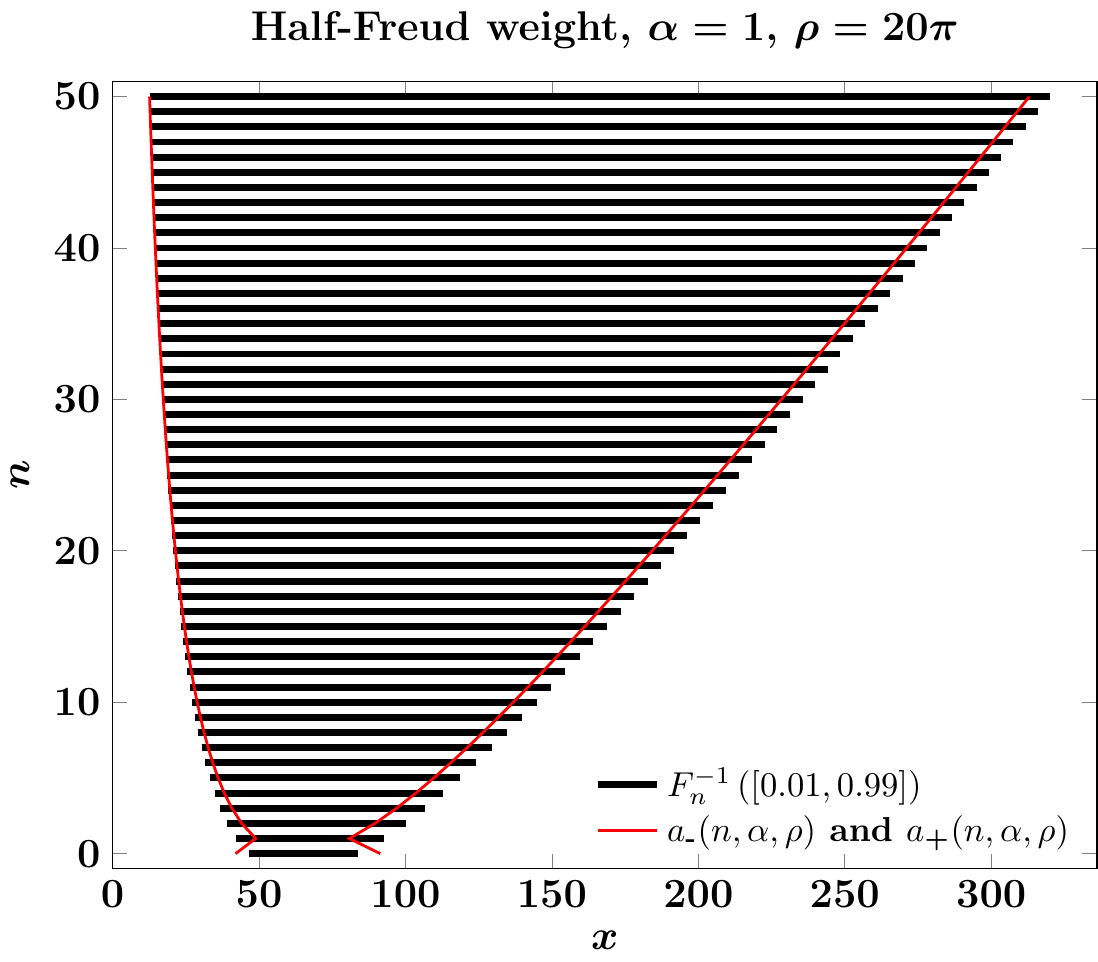}
    \includegraphics[width=0.49\textwidth]{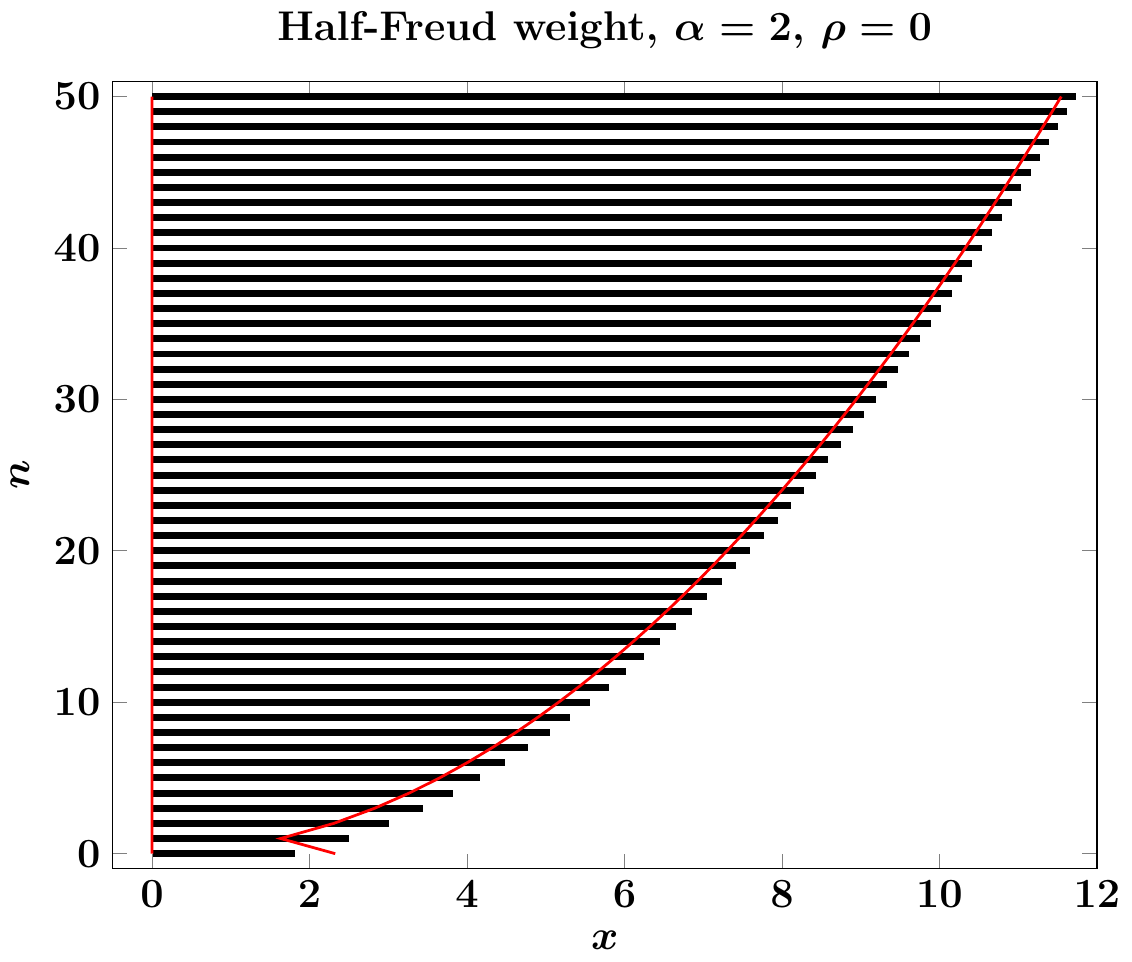}
  \end{center}
  \caption{Depiction of the interval $F_n^{-1}\left([0.01, 0.99]\right)$ containing ``most" of the support of $F_n(x)$, for various $n$ in the Half-Freud case. This interval is compared against the potential-theoretic demarcations \eqref{eq:half-freud-amin}. Good agreement indicates that $a_{\pm}$ defined in \eqref{eq:half-freud-interval} are reasonably accurate approximations for the bulk support of $F_n$.}\label{fig:half-freud-mrs}
\end{figure}

\subsubsection{Computing $\widehat{F}_n(x)$}
First assume that $x \leq x_0$. Then 
\begin{align*}
  F_n(x) &= \frac{1}{c_{HF}^{(\alpha,\rho)}} \int_0^x p_n^2(t) t^\rho \exp\left(-t^\alpha\right) \dx{t} \\
         &\leftstackrel{u = \frac{2 t}{x} - 1}{=} \left(\frac{x}{2}\right)^{\rho+1} \frac{1}{c_{HF}^{(\alpha,\rho)}} \int_{-1}^1 \exp\left(-\left(\frac{x}{2}\right)^\alpha \left(u+1\right)^\alpha\right) p_n^2\left(\frac{x}{2}\left(1+u\right)\right) (1+u)^\rho \dx{u}.
\end{align*}
We recognize a portion of the integrand as a Jacobi measure, and use successive measure modifications to define $\widetilde{\mu}_n$:
\begin{subequations}
\begin{align}\nonumber
  \dx{\widetilde{\mu}_n}(u) &\coloneqq p_n^2\left(\frac{x}{2}\left(1+u\right)\right) \dx{\mu_J}^{(0,\rho)} \\\label{eq:I-def-gfreud-left}
  %I(x) &\coloneqq \int_{-1}^1 \exp\left(-\left(\frac{x}{2}\right)^\alpha \left(u+1\right)^\alpha\right) \dx{\widetilde{\mu}_n^{(0,\rho)}(u)} \\\label{eq:Fn-I-gfreud-left}
  I(x) &\coloneqq \sum_{m=1}^M \widetilde{w}_m \exp\left(-\left(\frac{x}{2}\right)^\alpha \left(u_m+1\right)^\alpha\right)\\\label{eq:Fn-I-gfreud-left}
  \widehat{F}_n(x) &= \left(\frac{x}{2}\right)^{\rho+1} \frac{c_J^{(0,\rho)}}{c_{HF}^{(\alpha,\rho)}} I(x)
\end{align}
\end{subequations}
The recurrence coefficients of $\widetilde{\mu}_n$ can be computed via polynomial measure modifications on the roots $u_{j,n} = \frac{2 x_{j,n}}{x} - 1$, where $x_{j,n}$ are the roots of $p_n\left(\cdot\right)$. A detailed algorithm is given in Algorithm \ref{alg:Fn-GFreud-left}.

\begin{algorithm2e}[H]%need H when used with beamer
\SetKwInOut{Input}{input}\SetKwInOut{Output}{output}
\Input{$\alpha > \frac{1}{2}, \rho > -1$: half-line Freud weight parameters}
\Input{$n\in \N_0$ and $x \geq 0$: Order of induced measure $\mu_n$ and value $x$.}
\Input{$M \in \N$: Quadrature order for approximate computation of $F_n(x)$.}
\Output{$\widehat{F}_n(x)$}
\BlankLine
Compute $n$ zeros, $\left\{x_{j,n}\right\}_{j=1}^n$ of $p_n = p_n\left(\cdot; \mu^{(a,\rho)}\right)$, and leading coefficient $\gamma_n$ of $p_n$.\;
Compute recurrence coefficients $a_j$ and $b_j$ associated to $\mu_J^{(0,\rho)}$ for $0 \leq j \leq M + 2 n$.\;
\For{$j=1, \ldots, n$}{
  Quadratic measure modification \eqref{eq:quadratic-modification}: update $a_n$ and $b_n$ for $n = 0, \ldots, M + 2(n-j)$ with modification factor $\left(u - \left(\frac{2 x_{j,n}}{x} - 1\right) \right)^2$.\\
  %Use quadratic modification \eqref{eq:quadratic-modification} to update $a_n$ and $b_n$ for $n = 0, \ldots, M + 2(n-j)$.\\
  %Update recurrence coefficients $a_n$ and $b_n$ for $n = 0, \ldots, M + 2(n-j)$ via modification:\\
  %\hskip 10pt If $u_{j,n} = \frac{2 x_{j,n}}{x} - 1 \in [-1,1]$, then use quadratic modification \eqref{eq:quadratic-modification}.\\
  %\hskip 10pt If $u_{j,n} \not\in [-1,1]$, then use twice-application of linear modification \eqref{eq:linear-modification}.\\
  %Algorithm \ref{alg:quadratic-modification}: use quadratic modification at root $u_{j,n} = \frac{2 x_{j,n}}{x} - 1$ to update $a_k$ and $b_k$ for $0 \leq k \leq M + 2(n-j)$.\;
  Scale $b_0 \gets b_0 \exp\left(\frac{1}{n} \log \gamma_n^2 \right)$.
}
%Compute $M$-point Gauss quadrature associated with measure $\widetilde{\mu}^{(0,\rho)}_n$ via $\left\{\left(a_j, b_j\right)\right\}_{j=0}^M$.\;
%Approximate the integral $I$ in \eqref{eq:I-def-gfreud-left} with the $M$-point quadrature.\;
%Evaluate \eqref{eq:Fn-I-gfreud-left}.\;
Compute $M$-point Gauss quadrature $\left(\widetilde{u}_m, \widetilde{w}_m\right)_{m=1}^M$ associated with measure $\widetilde{\mu}_n$ via $\left\{\left(a_j, b_j\right)\right\}_{j=0}^M$.\\
Compute the integral $I$ in \eqref{eq:I-def-gfreud-left}, and return $\widehat{F}_n\left(x; \mu^{(\alpha,\rho)}\right)$ given by \eqref{eq:Fn-I-gfreud-left}.
\caption{Computation of $F_n(x)$ for $\mu = \mu_{HF}^{(\alpha,\rho)}$ corresponding to a half-line Freud weight.}\label{alg:Fn-GFreud-left}
\end{algorithm2e}

Now assume that $x \geq x_0$. We compute $F_n^c$ directly in a similar fashion as we did for $F_n$. We have
\begin{align*}
  F^c_n(x) &= \frac{1}{c_{HF}^{(\alpha, \rho)}} \int_x^{\infty} p_n^2\left(t\right) t^\rho \exp\left(-t^\alpha\right) \dx{t} \\
           &\leftstackrel{u=t-x}{=} \frac{1}{c_{HF}^{(\alpha, \rho)}} \int_0^{\infty} p_n^2\left(u+x\right) \left(u+x\right)^\rho \exp\left(-\left(u + x\right)^\alpha\right) \dx{u} \\
           &= \exp(-x^\alpha) \frac{c_{HF}^{(\alpha,0)}}{c_{HF}^{(\alpha, \rho)}} \int_0^{\infty} p_n^2\left(u+x\right) \left(u+x\right)^\rho \exp\left(u^\alpha + x^\alpha - \left(u + x\right)^\alpha\right) \exp(-u^\alpha)\dx{u}
\end{align*}
We again use this to define a new measure $\widetilde{\mu}_n$ and an associated $M$-point Gauss quadrature $\left(u_m, w_m\right)_{m=1}^M$. The recurrence coefficients for $\widetilde{\mu}_n$ are computable via polynomial measure modifications. This results in the approximation
\begin{align}
  \dx{\widetilde{\mu}_n}(u) &\coloneqq p_n^2\left(u+x\right) \dx{\mu_{HF}}^{(\alpha,0)} \\\label{eq:I-def-gfreud}
  I(x)     &= \sum_{m=1}^M \widetilde{w}_m \left(u_m + x\right)^\rho \exp\left(u_m^\alpha + x^\alpha - (u_m + x)^\alpha\right)\\\label{eq:Fn-I-gfreud}
  \widehat{F}^c_n(x) &= \exp(-x^\alpha)\frac{c_{HF}^{(\alpha, 0)}}{c_{HF}^{(\alpha, \rho)}} I(x)
\end{align}
%Let's assume $\alpha = 1$: then we have 
%\begin{align*}
%  F^c_n(x) &= \frac{\exp(-x)}{c_L^{1, \rho}} \int_0^{\infty} p_n^2\left(u+x\right) \left(u+x\right)^\rho \exp\left(-u\right) \dx{u} \\
%           &= \frac{\exp(-x) c_L^{1, 0}}{c_L^{1, \rho}} \int_0^{\infty} p_n^2\left(u+x\right) \left(u+x\right)^\rho \dx{\mu^{\ast(1,0)}}(u)
%\end{align*}
%The roots $x_{j,n}$, $j=1, \ldots, n$ of $p_n(\cdot)$ can be used to compute the roots of $p_n(\cdot+x)$, given by $u_{j,n} = x_{j,n} - x$. %Now define $\dx{\widetilde{\mu}}_n^{(\ast(1,0)} \coloneqq p_n^2(u+x) \dx{\mu^{\ast(1,0)}}$, the computation of which can be accomplished by successive quadratic modification as described in the Jacobi case. Then
%\begin{align}\label{eq:I-def-gfreud}
%  I(x)     &= \int_0^{\infty} \left(u+x\right)^\rho \dx{\widetilde{\mu}_n^{\ast(1,0)}}(u) \\\label{eq:Fn-I-gfreud}
%  F^c_n(x) &= \frac{c_L^{1, 0}}{c_L^{1, \rho}} e^{-x} I(x) = 1 - \frac{1}{e^x \Gamma\left(\rho+1\right)} I(x)
%\end{align}
%We accomplish computation of $I(x)$ via $M$-point $\widetilde{\mu}_n^{\ast(1,0)}$-Gaussian quadrature. 
A more detailed algorithm is given in Algorithm \ref{alg:Fn-GFreud}. Of course, once $F_n^c$ is computed we may compute $F_n(x) = 1 - F_n^c(x)$.

%Worksheet:
%\begin{align*}
%  \frac{c_{F\ast}^{(1,0)}}{c_{F\ast}^{(1,\rho)}} = \frac{\Gamma(1)}{\Gamma(\rho+1)} = \frac{1}{\Gamma(\rho)}.
%\end{align*}

\begin{algorithm2e}[H]%need H when used with beamer
\SetKwInOut{Input}{input}\SetKwInOut{Output}{output}
\Input{$\alpha > \frac{1}{2}, \rho > -1$: generalized Freud weight parameters}
\Input{$n\in \N_0$ and $x \geq 0$: Order of induced polynomial and measure $\mu_n$ and value $x$.}
\Input{$M \in \N$: Quadrature order for approximate computation of $F^c_n(x)$.}
\Output{$\widehat{F}^c_n(x)$}
\BlankLine
Compute $n$ zeros, $\left\{x_{j,n}\right\}_{j=1}^n$ of $p_n = p_n\left(\cdot; \mu^{(a,\rho)}\right)$, and leading coefficient $\gamma_n$ of $p_n$.\;
Compute recurrence coefficients $a_j$ and $b_j$ associated to $\mu_{HF}^{(\alpha,0)}$ for $0 \leq j \leq M + 2 n$.\;
\For{$j=1, \ldots, n$}{
  %Use \eqref{eq:quadratic-modification} to update $a_n$ and $b_n$ for $n = 0, \ldots, M + 2(n-j)$.\\
  Quadratic measure modification \eqref{eq:quadratic-modification}: update $a_n$, $b_n$ for $n=0, \ldots, M + A + 2(n-j)$ with modification factor $\left(u - \left( x_{j,n} - x \right)\right)^2$.\\
  %Update recurrence coefficients $a_n$ and $b_n$ for $n = 0, \ldots, M + 2(n-j)$ via modification:\\
  %\hskip 10pt If $u_{j,n} = x_{j,n} - x \geq 0$, then use quadratic modification \eqref{eq:quadratic-modification}.\\
  %\hskip 10pt If $u_{j,n} < 0$, then use twice-application of linear modification \eqref{eq:linear-modification}.\\
  Scale $b_0 \gets b_0 \exp\left(\frac{1}{n} \log \gamma_n^2 \right)$.
}
%Compute $M$-point Gauss quadrature associated with measure $\widetilde{\mu}_n$ via $\left\{\left(a_j, b_j\right)\right\}_{j=0}^M$.\;
%Approximate the integral $I$ in \eqref{eq:I-def-gfreud} with the $M$-point quadrature.\;
%Evaluate \eqref{eq:Fn-I-gfreud}.\;
Compute $M$-point Gauss quadrature $\left(\widetilde{u}_m, \widetilde{w}_m\right)_{m=1}^M$ associated with measure $\widetilde{\mu}_n$ via $\left\{\left(a_j, b_j\right)\right\}_{j=0}^M$.\\
Compute the integral $I$ in \eqref{eq:I-def-gfreud}, and return $\widehat{F}^c_n\left(x; \mu^{(\alpha,\rho)}\right)$ given by \eqref{eq:Fn-I-gfreud}.
\caption{Computation of $\widehat{F}^c_n(x)$ for $\mu^{(\alpha,\rho)}_{HF}$ corresponding to a half-line Freud weight.}\label{alg:Fn-GFreud}
\end{algorithm2e}

Errors between $F_n$ and computational approximations $\widehat{F}_n$ are shown in Figure \ref{fig:half-freud-error}. We see that we require a much larger value of $M$ in order to achieve accurate approximations compared to the Jacobi case. We believe this to be the case due to the function $\exp(u^\alpha + x^\alpha - (u+x)^\alpha)$ appearing in the integral $I(x)$. Note that for $\alpha = 1$ this function becomes unity and so does not adversely affect the integral; this results in the much more favorable error plot on the left in Figure \ref{fig:half-freud-error}. 

The different behavior for $\alpha = 1$ leads us to make customized choices in this case: we choose $M = 25$ for all values of $n$ and $\rho$, and we take $x_0(n) \equiv 50$.  Whereas for $\alpha \neq 1$ our tests suggest that $M = n + 10$ is sufficient to achieve good accuracy, and we take $x_0$ as the average of $a_\pm$ as given in \eqref{eq:half-freud-interval}.

\begin{figure}
  \begin{center}
    \includegraphics[width=0.49\textwidth]{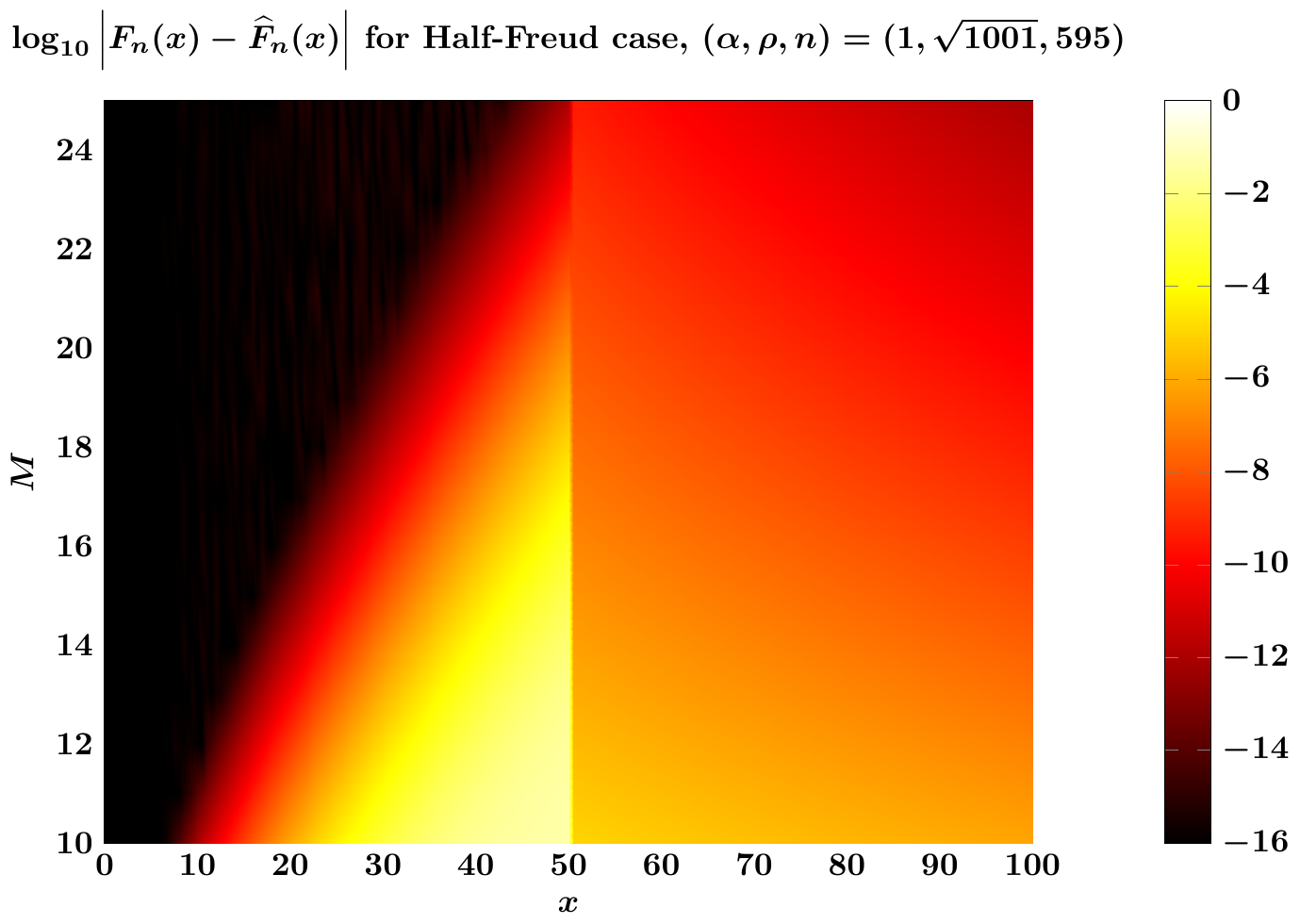}
    \includegraphics[width=0.49\textwidth]{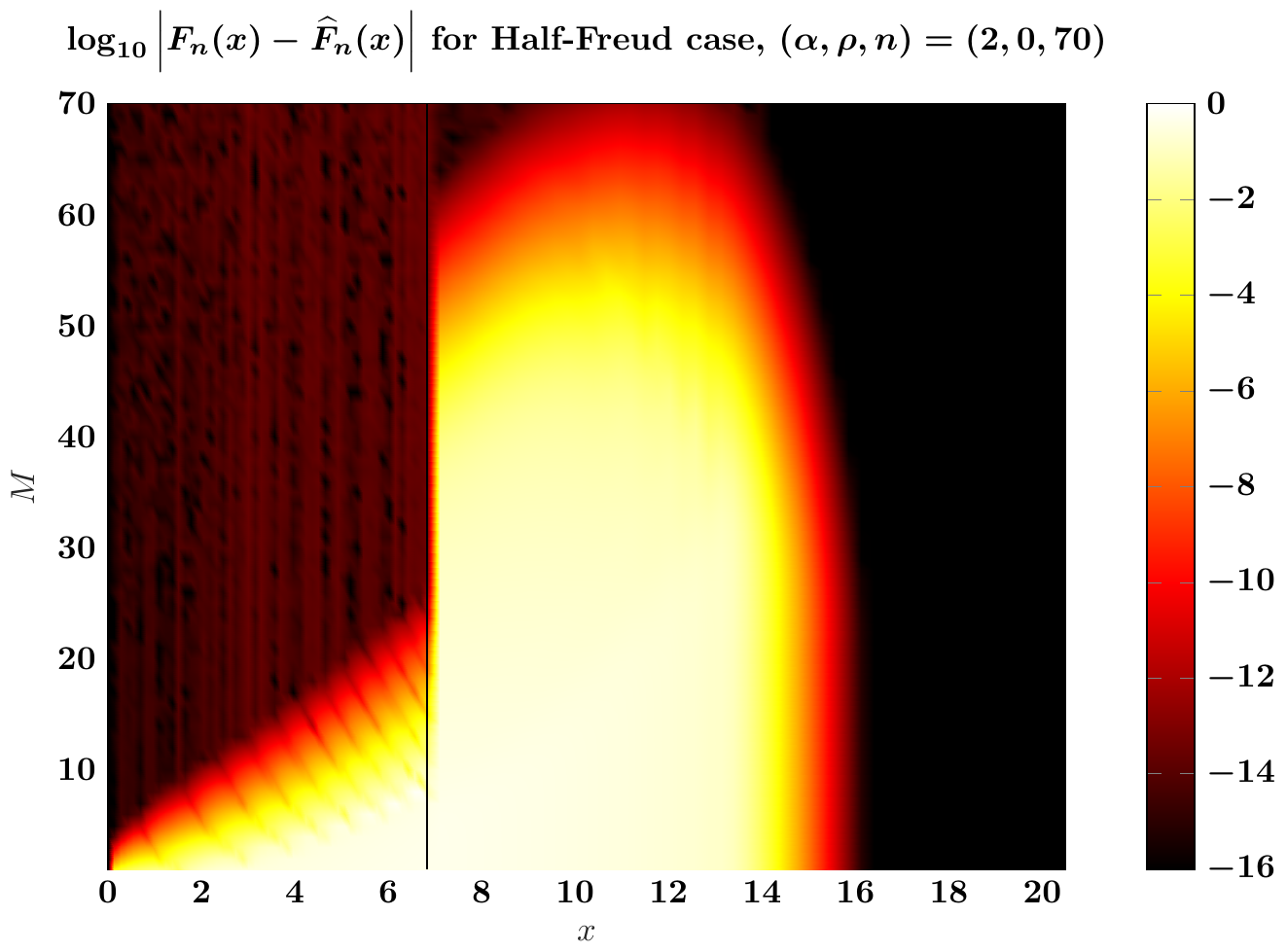}
  \end{center}
  \caption{Color plot of $\log_{10} \left| F_n(x) - \widehat{F}_n(x) \right|$ for half-line Freud measures with certain choices of the order $n$ and various quadrature size $M$. Left: $(\alpha, \rho, n) = \left(1, \sqrt{1001}, 595\right)$. Right: $(\alpha, \rho, n) = \left(2, 0, 70\right)$. Much larger values of $M$ are needed when compared against the Jacobi case in Figure \ref{fig:jacobi-error}.}\label{fig:half-freud-error}
\end{figure}

\subsection{Freud weights}
Finally, consider the Freud measure $\mu_F^{(\alpha,\rho)}$ defined in Table \ref{tab:measures}. 
%$\mu$ be a measure supported on $\R$ whose density is 
%\begin{align}\label{eq:freud-weight}
%  \dx{\mu}_F(x) &= \frac{1}{c_F^{(a,\rho)}} |x|^\rho \exp\left( - |x|^a \right), & a > 1, \hskip 5pt \rho > -1
%\end{align}
%where $c_F^{(\rho, a)}$ is a normalization factor, equal to 
%\begin{align}\label{eq:freud-normalization}
%  c_F^{(a, \rho)} = \frac{2}{a} \Gamma \left(\frac{\rho + 1}{a} \right).
%\end{align}
An especially important case occurs for $\alpha= 2$, $\rho = 0$, corresponding to the classical Hermite polynomials. Note that the recurrence coefficients for general values of $a$ and $\rho$ are not known explicitly, but their asymptotic behavior has been established \cite{lubinsky_proof_1988}.

It is well-known that Freud weights are essentially half-line Freud weights in disguise under a quadratic map. It is not surprising then that we may express primitives of induced polynomial measures for Freud weights in terms of the associated half-line Freud primitives.
\begin{theorem}\label{thm:freud-is-gfreud}
  Let parameters $(\alpha,\rho)$ define a Freud weight and associated measure $\mu_F^{(\alpha,\rho)}$. Then for $x \leq 0$, 
  \begin{align}\label{eq:freud-is-gfreud}
    F_n\left(x; \mu_F^{(\alpha,\rho)}\right) = \left\{ \begin{array}{rl} 
        \frac{1}{2} F_{n/2}^c\left( x^2; \mu_{HF}^{(\alpha/2,(\rho-1)/2)} \right), & n \textrm{ even} \\
        \frac{1}{2} F_{(n-1)/2}^c \left(x^2; \mu_{HF}^{(\alpha/2, (\rho+1)/2)} \right), & n \textrm{ odd}
      \end{array}\right.
  \end{align}
  For $x \geq 0$, we have
  \begin{align}\label{eq:freud-primitive-is-even}
    F_n\left(x; \mu_F^{(\alpha,\rho)}\right) = 1 - F_n\left(-x; \mu_F^{(\alpha,\rho)}\right).
  \end{align}
\end{theorem}
Note that with expressions \eqref{eq:freud-is-gfreud} and \eqref{eq:freud-primitive-is-even}, an algorithm for computing $F_n\left(\cdot; \mu_F\right)$ is straightforward to devise utilizing Algorithm \ref{alg:Fn-GFreud} for $F_n^c\left(\cdot; \mu_{HF}\right)$.

The result \eqref{eq:freud-primitive-is-even} follows easily from the fact that the integrand in \eqref{eq:mun-def} defining $F_n$ is an even function. To prove the main portion of the theorem, expression \eqref{eq:freud-is-gfreud}, we require the following result relating Freud orthonormal polynomials to half-line Freud orthonormal polynomials.
\begin{lemma}\label{lemma:hermite-laguerre-relation}
  Let $\rho > -1$ and $\alpha > 1$ be parameters that define a Freud measure $\mu^{(\alpha,\rho)}_F$ with associated orthonormal polynomial family $p_n(x) = p_n\left(x; \mu_F^{(\alpha,\rho)}\right)$. Define two sets of half-line Freud parameters $\left(\alpha_\ast, \rho_\ast\right)$ and $\left( \alpha_{\ast\ast}, \rho_{\ast\ast}\right)$ and the corresponding half-line Freud measures and polynomials:
  \begin{subequations}\label{eq:all-ast-def}
  \begin{align}\label{eq:ast-def}
    \alpha_\ast &\coloneqq \frac{\alpha}{2}, & \rho_\ast &\coloneqq \frac{\rho-1}{2}, & p_{\ast,n}(x) \coloneqq p_{n}\left(x; \mu_{HF}^{(\alpha_\ast,\rho_\ast)}\right),\\\label{eq:astast-def}
    %\int_0^\infty p_{\ast, n}(x) p_{\ast,m}(x) \dx{\mu_{F\ast}^{\left(\alpha_\ast, \rho_\ast \right)}}(x)&= \delta_{m,n} 
    \alpha_{\ast\ast} &\coloneqq \frac{\alpha}{2}, & \rho_{\ast\ast} &\coloneqq \frac{\rho+1}{2}, & p_{\ast\ast,n}(x) \coloneqq p_{n}\left(x; \mu_{HF}^{(\alpha_{\ast\ast},\rho_{\ast\ast})}\right)
    %\int_0^\infty p_{\ast\ast,n}(x) p_{\ast\ast,m}(x) \dx{\mu_{F\ast}^{\left(\alpha_{\ast\ast}, \rho_{\ast\ast}\right)}}(x) &= \delta_{m,n}
  \end{align}
  \end{subequations}
  Also define the constant
  \begin{align}\label{eq:h-def}
    h^2 = h^2(\alpha,\rho) \coloneqq \frac{\Gamma\left( \frac{\rho + 1}{\alpha}\right)}{\Gamma\left( \frac{\rho+3}{\alpha}\right)}
  \end{align}
  Then, for all $n \geq 0$, 
  \begin{align*}
    p_{2 n}\left(x\right) &= p_{\ast,n} \left( x^2 \right), \\
    p_{2 n+1}\left(x\right) &= h x p_{\ast\ast,n} \left( x^2 \right),
  \end{align*}
  %The family $\left\{p_n^\ast\right\}_{n=0}^\infty$ is the orthonormal polynomial family associated to a generalized Freud measure $\mu^\ast$ with parameters $\alpha^\ast = \alpha/2$ and $\rho^\ast = \frac{\rho - 1}{2}$. The family $\left\{p_n^{\ast\ast}\right\}_{n=0}^\infty$ is the orthonormal polynomial family associated to a generalized Freud measure $\mu^{\ast\ast}$ with parameters $\alpha^{\ast\ast} = \alpha/2$ and $\rho^{\ast\ast} = \frac{\rho + 1}{2}$.
\end{lemma}
\begin{proof}
  The following equalities may be verified via direct computation using the definitions \eqref{eq:all-ast-def} and \eqref{eq:h-def} along with the expressions in Table \ref{tab:measures}:
  \begin{align}\label{eq:cF-cFast-relations}
    c_F^{(\alpha,\rho)} &= c_{HF}^{(\alpha_\ast,\rho_\ast)}, & 
    c_{F}^{(\alpha,\rho)} &= h^2 c_{HF}^{(\alpha_{\ast\ast},\rho_{\ast\ast})}
  \end{align}
  The proof of this lemma relies on the change of measure $t \mapsto x^2$. We have
  \begin{align*}
    \delta_{m,n} &= \int_0^\infty p_{\ast,n}(t) p_{\ast,m}(t) \dx{\mu_{HF}^{\left(\alpha_\ast,\rho_\ast\right)}(t)} = \frac{1}{c_{HF}^{(\alpha_\ast,\rho_\ast)}} \int_0^\infty p_{\ast,n} (t) p_{\ast,m}(t) t^{\rho_\ast} \exp\left(-t^{\alpha_\ast}\right) \dx{t} \\
                 &= \frac{2}{c_{HF}^{(\alpha_\ast,\rho_\ast)}} \int_0^\infty p_{\ast,n}\left(x^2\right) p_{\ast,m}\left(x^2\right) x^{2 \rho_\ast+1} \exp\left(-x^{2\alpha_\ast}\right) \dx{x} \\
                 &= \frac{1}{c_{F}^{(\alpha,\rho)}} \int_{-\infty}^\infty p_{\ast,n}\left(x^2\right) p_{\ast,m}\left(x^2\right) \left|x\right|^{2 \rho_\ast+1} \exp\left(-\left|x\right|^{2\alpha_\ast}\right) \dx{x} \\
                 &= \int_{-\infty}^{\infty} p_{\ast,n}\left(x^2\right) p_{\ast,m}\left(x^2\right) \dx{\mu_F^{(\alpha,\rho)}}(x).
  \end{align*}
  This relation shows that the family $\left\{p_{\ast,n}\left(x^2\right)\right\}_{n=0}^\infty$ are polynomials of degree $2 n$ that are orthonormal under a Freud weight with parameters $\alpha = 2\alpha_\ast$ and $\rho = 2 \rho_\ast + 1$. Using nearly the same arguments, but with the family $p_{\ast\ast,n}$, yields the relation
  \begin{align*}
    \delta_{m,n} &= \int_0^\infty p_{\ast\ast,n}(t) p_{\ast\ast,m}(t) \dx{\mu_{HF}^{\left(\alpha_{\ast\ast},\rho_{\ast\ast}\right)}(t)} \\
                 &= \frac{1}{c_{HF}^{(\alpha_{\ast\ast},\rho_{\ast\ast})}} \int_0^\infty p_{\ast\ast,n}(t) p_{\ast\ast,m}(t) t^{\rho_{\ast\ast}} \exp\left(-t^{\alpha_{\ast\ast}}\right) \dx{t} \\
                 &= \frac{2}{c_{HF}^{(\alpha_{\ast\ast},\rho_{\ast\ast})}} \int_0^\infty p_{\ast\ast,n}\left(x^2\right) p_{\ast\ast,m}\left(x^2\right) x^{2 \rho_{\ast\ast}+1} \exp\left(-x^{2\alpha_{\ast\ast}}\right) \dx{t} \\
                 &= \frac{h^2}{c_F^{(\alpha,\rho)}} \int_{-\infty}^\infty x p_{\ast\ast,n}\left(x^2\right)\; x p_{\ast\ast,m}\left(x^2\right) \left|x\right|^{2 \rho_{\ast\ast}-1} \exp\left(-\left|x\right|^{2\alpha_{\ast\ast}}\right) \dx{t} \\
                 &= h^2 \int_{-\infty}^{\infty} x p_{\ast\ast,n}\left(x^2\right)\; x p_{\ast\ast,m}\left(x^2\right) \dx{\mu_F^{(\alpha,\rho)}}(x).
  \end{align*}
  This relation shows that the family $\left\{h x p_{\ast\ast,n}\left(x^2\right)\right\}_{n=0}^\infty$ are polynomials of degree $2 n+1$ that are orthonormal under the same Freud having parameters $\alpha = 2\alpha_{\ast\ast}$ and $\rho = 2 \rho_{\ast\ast} - 1$. We also have that $x p_{\ast\ast,n}\left(x^2\right)$ is orthogonal to $p_{\ast,m}\left(x^2\right)$ under a(ny) Freud weight for any $n, m$ because of even-odd symmetry. Thus, define 
  \begin{align*}
    P_n(x) = \left\{ \begin{array}{rl} p_{\ast,n/2}\left(x^2\right), & n \textrm{ even} \\
    h x p_{\ast\ast,(n-1)/2}\left(x^2\right), & n \textrm{ odd}
  \end{array}\right.
  \end{align*}
  Then $\left\{P_n\right\}_{n=0}^\infty$ is a family of degree-$n$ polynomials (with positive leading coefficient) orthonormal under a $(\alpha, \rho)$ Freud weight. Therefore, $P_n \equiv p_n$.
\end{proof}

We can now give the 
\begin{proof}[Proof of Thereom \ref{thm:freud-is-gfreud}]
Assume $x \leq 0$. Then 
\begin{align*}
  F_n\left(x; \mu_F^{(\alpha,\rho)}\right) &= \frac{1}{c_F^{(a, \rho)}} \int_{-\infty}^x p_n^2\left(t\right) |t|^{\rho} \exp\left(- |t|^\alpha\right) \dx{t},  \\
         &\leftstackrel{u = t^2}{=} \frac{1}{2 c_F^{(a, \rho)}} \int_{x^2}^{\infty} p_n^2\left( -\sqrt{u}\right) |u|^{\rho_\ast} \exp\left(-|u|^{\alpha_\ast}\right) \dx{u},
\end{align*}
where $(\alpha_\ast, \rho_\ast)$ are as defined in \eqref{eq:ast-def}. By Lemma \ref{lemma:hermite-laguerre-relation}, 
\begin{align*}
    p^2_n(-\sqrt{u}) = \left\{ \begin{array}{rl} p^2_{\ast,n/2}\left(u\right), & n \textrm{ even} \\
    h^2 u p^2_{\ast\ast,(n-1)/2}\left(u\right), & n \textrm{ odd}
\end{array}\right.
\end{align*}
  where $(\alpha_{\ast\ast}, \rho_{\ast\ast})$ are defined in \eqref{eq:astast-def}. Then if $n$ is even, we have
\begin{align*}
  F_n\left(x; \mu_F^{(\alpha_\ast,\rho_\ast)}\right) &=  \frac{1}{2 c_{HF}^{(\alpha_\ast, \rho_\ast)}} \int_{x^2}^{\infty} p^2_{\ast,n/2}(u) |u|^{\rho_\ast} \exp\left(-|u|^{\alpha_\ast}\right) \dx{u}, \\
                                                    &= \frac{1}{2} F_{n/2}^c\left( x^2; \mu_{HF}^{(\alpha_\ast,\rho_\ast)} \right),
\end{align*}
where we recall the equalities \eqref{eq:cF-cFast-relations} related $c_F$ to $c_{F\ast}$. Similarly, if $n$ is odd we have
\begin{align*}
  F_n\left(x; \mu_F^{(\alpha,\rho)}\right) &=  \frac{h^2}{2 c_F^{(\alpha, \rho)}} \int_{x^2}^{\infty} u p^2_{\ast\ast,(n-1)/2}(u) |u|^{\rho_\ast} \exp\left(-|u|^{\alpha_\ast}\right) \dx{u} \\
                                           &=  \frac{1}{2 c_{HF}^{(\alpha_{\ast\ast}, \rho_{\ast\ast})}} \int_{x^2}^{\infty} p^2_{\ast\ast,(n-1)/2}(u) |u|^{\rho_{\ast\ast}} \exp\left(-|u|^{\alpha_{\ast\ast}}\right) \dx{u} \\
         &= \frac{1}{2} F_{(n-1)/2}^c \left(x^2; \mu_{HF}^{(\alpha_{\ast\ast}, \rho_{\ast\ast})} \right)
\end{align*}
The combination of these results proves \eqref{eq:freud-is-gfreud}.
\end{proof}

\section{Inverting induced distributions}\label{sec:invert-induced}
We have discussed at length in previous sections algorithms for computing $F_n(x)$ defined in \eqref{eq:mun-def} for various Lebesgue-continuous measures $\mu$. The central application of these algorithms we investigate in this paper is actually in the evaluation of $F_n^{-1}(u)$ for $u \in [0,1]$. We accomplish this by solving for $x$ in the equation
\begin{align}\label{eq:rootfind}
  F_n(x) - u &= 0, & u &\in [0,1],
\end{align}
using a root-finding method. Our first step involves providing an initial guess for $x$.

\subsection{Computing an initial interval}\label{sec:ms-interval}
We use $s_{\pm}$ to denote the (possibly infinite) endpoints of the support of $\mu$:
\begin{align*}
  -\infty \leq s_- &\coloneqq \inf \left(\supp \mu\right), & s_+ &\coloneqq \sup \left(\supp \mu\right) \leq \infty.
\end{align*}
Now let $u \in [0,1]$. Our first step in finding $F_{n}^{-1}(u)$ is to compute two values $x_-$ and $x_+$ such that 
\begin{align}\label{eq:root-sandwich}
  x_- \leq F^{-1}_n(u) \leq x_+.
\end{align}
Our procedure for identifying an initial interval containing $F_n^{-1}(u)$ leverages the Markov-Stiltjies inequalities for orthogonal polynomials. These inequalities state that empirical probability distributions of Gauss quadrature rules generated from a measure bound the distribution function for this measure. Precisely:
\begin{lemma}[Markov-Stiltjies Inequalties, \cite{szego_orthogonal_1975}]
  Let $\mu$ be a probability measure on $\R$ with an associated orthogonal polynomial family. For any $N \in \N$, let $\left\{ x_{k,N}, w_{k,N} \right\}_{k=1}^N$ denote the $N$ $\mu$-Gaussian quadrature nodes and weights, respectively. Then:
  \begin{align*}
    \sum_{k=1}^{m-1} w_{k,N} \leq &F\left(x_{m,N}\right) \leq \sum_{k=1}^m w_{k,N}, & 1 &\leq m \leq N.
  \end{align*}
\end{lemma}
Let $\left\{p_{j,n}\right\}_{j=0}^\infty$ denote the sequence of polynomial orthonormal under the induced measure $\mu_n$. Given $N \in \N$, let $\left\{z_{k,N}, v_{k,N} \right\}_{k=1}^N$ denote the $N$-point $\mu_n$-Gaussian quadrature rule, i.e., $z_{k,N}$ are the $N$ ordered zeros of $p_{N,n}(x)$, with $v_{k,N}$ the associated weights. Since $\mu_n$ is a probability measure, then $\sum_{k=1}^N v_{k,N} = 1$. As such, given $u \in [0,1]$ we can always find some $m \in \left\{1, \ldots, N\right\}$ such that 
\begin{align}\label{eq:u-markov-stiltjies-sandwich}
  \sum_{k=1}^{m-1} v_{k,N} \leq u \leq \sum_{k=1}^m v_{k,N}.
\end{align}
Then, defining $z_{0,N} \equiv s_-$ and $z_{N+1,N} \equiv s_+$ for all $N$ and $n$, we have
\begin{align*}
  F_n\left(z_{m-1,N}\right) \leq \sum_{k=1}^{m-1} v_{k,N} \leq u \leq \sum_{k=1}^m v_{k,N} \leq F_n\left(z_{m+1,N}\right)
\end{align*}
Since $F_n$ is non-decreasing, this is equivalently,
\begin{align*}
  z_{m-1,N} \leq F_n^{-1}\left(u\right) \leq z_{m+1,N}
\end{align*}
Thus, if we find an $m$ such that \eqref{eq:u-markov-stiltjies-sandwich} holds, then \eqref{eq:root-sandwich} holds with
\begin{align}\label{eq:initial-guess}
  x_- &= z_{m-1,N}, & x_+ &= z_{m+1, N}
\end{align}
When $\mathrm{supp}\;\mu$ is bounded, the $N$-asymptotic density of orthogonal polynomial zeros on $\mathrm{supp}\;\mu$ guarantees that we can find a bounding interval with endpoints $x_{\pm}$ of arbitrarily small width by taking $N$ sufficiently large. The difficulty is that we therefore require the zeros $z_{k,N}$ and the quadrature weights $v_{k,N}$ of the induced measure, which in turn require knowledge of the three-term recurrence coefficients associated to $\mu_n$. These can be easily computed from the coefficients associated to $\mu$; since
\begin{align}\label{eq:induced-measure-breakdown}
  \dx{\mu}_n(x) = p_n^2(x) \dx{\mu}(x) = \gamma_n^2 \prod_{j=1}^n \left(x - x_{j,n} \right)^2 \dx{\mu}(x),
\end{align}
then we may again iteratively utilize the quadratic modification algorithm given by \eqref{eq:quadratic-modification} to compute these recurrence coefficients, which are iteratively quadratic modifications of the $\mu$-coefficients. (Note that this is precisely the procedure proposed in \cite{gautschi_set_1993} for computing these coefficients.)

\subsection{Bisection}
For simplicity, the root-finding method we employ to solve \eqref{eq:rootfind} is the bisection approach. More sophisticated methods may be used, with the caveat that the derivative of the function, $F'(x) = p_n^2(x) \dx{\mu}(x)$, vanishes wherever $p_n$ has a root. We have found that a naive application of Newton's method for root-finding often runs into trouble, even with a very accurate initial guess.

The bisection method for root-finding applied to \eqref{eq:rootfind} starts with an initial guess for an interval $[x_-, x_+]$ containing the root $x$, and iteratively updates this interval via
\begin{align*}
  x_- \gets \frac{1}{2} \left(x_- + x_+\right) \hskip 15pt &\textrm{ if } \hskip 15pt F_n\left(\frac{1}{2} \left(x_- + x_+\right) \right) \leq u \\
  x_+ \gets \frac{1}{2} \left(x_- + x_+\right) \hskip 15pt &\textrm{ if } \hskip 15pt F_n\left(\frac{1}{2} \left(x_- + x_+\right) \right) > u
\end{align*}
After a sufficient number of iterations so that $x_+ - x_-$ and/or $|F(x_-) - F(x_+)|$ is smaller than a tunable tolerance parameter, one can confidently claim to have found the root $x$ to within this tolerance. A good initial guess for $x_{\pm}$ lessens the number of evaluations of $F_n$ in a bisection approach and thus accelerates overall evaluation of $F_n^{-1}$. 

The overall algorithm for solving \eqref{eq:rootfind} is to (i) compute the recurrence coefficients associated with $\mu_n$ in \eqref{eq:induced-measure-breakdown} via quadratic measure modifications, (ii) compute order-$N$ $\mu_n$-Gaussian quadrature nodes and weights $z_{j,N}$ and $v_{j,N}$, respectively, (iii) identify $m$ such that \eqref{eq:u-markov-stiltjies-sandwich} holds so that $x_{\pm}$ may be computed in \eqref{eq:initial-guess}, and (iv) iteratively apply the bisection algorithm with the initial interval defined by $x_{\pm}$ using the evaluation procedures for $F_n$ outlined in Section \ref{sec:Fn-eval}.

\section{Applications}\label{sec:applications}
This section discusses two applications of sampling from univariate induced measures. Both these applications consider multivariate scenarios, and are based on the fact that many ``interesting" multivariate sampling measures are additive mixtures of tensorized univariate induced measures. Our first task is to introduce notation for tensorized orthogonal polynomials.

We will write $d$-variate polynomials using multi-index notation: $\lambda \in \N_0^d$ denotes a multi-index with components $\lambda = \left(\lambda_1, \ldots, \lambda_d\right)$ and magnitude $|\lambda| = \sum_{j=1}^d \lambda_j$. A point $x \in \R^d$ has components $x = \left(x_1, \ldots, x_d \right)$, and $x^\lambda = \prod_{j=1}^d x_j^{\lambda_j}$. A collection of multi-indices will be denoted $\Lambda \subset \N_0^d$; we will assume that $N = \left|\Lambda\right|$ is finite. 

Let $\mu$ be a tensorial measure on $\R^d$ such that each of its $d$ marginal univariate measures $\mu^{(j)}$, $j=1, \ldots, d$ admits a $\mu^{(j)}$-orthonormal polynomial family $\left\{ p_{j,n} \right\}_{n=0}^\infty$ on $\R$, satisfying
\begin{align*}
  \int_\R p_{j,n}\left(x_j\right) p_{j,m}\left(x_j\right) \dx{\mu^{(j)}}\left(x_j\right) &= \delta_{m,n}, & n, m& \in \N_0,\;\; j = 1, \ldots, d.
\end{align*}
A tensorial $\mu$ allows us to explicitly construct an orthonormal polynomial family for $\mu$ from univariate polynomials,
\begin{align*}
  p_\lambda(x) \coloneqq \prod_{j=1}^d p_{j, \lambda_j}\left(x_j\right).
\end{align*}
These polynomials are an $L^2_{\dx{\mu}}$-orthonormal basis for the subspace $P_\Lambda$, defined as
\begin{align*}
  P_\Lambda = \mathrm{span} \left\{ p_\lambda \;\; | \;\; \lambda \in \Lambda \right\}.
\end{align*}
Under the additional assumption that the index set $\Lambda$ is downward-closed, then $P_\Lambda = \mathrm{span} \left\{ x^\lambda \;\; |\;\; \lambda \in \Lambda\right\}$.

We extend our definition of induced polynomials to this tensorial multivariate situation. For any $\lambda \in \Lambda$, the order-$\lambda$ \textit{induced} measure $\mu_\lambda$ is defined as 
\begin{align*}
  \dx{\mu}_\lambda(x) \coloneqq p_\lambda^2(x) \dx{\mu}(x) = \prod_{j=1}^d p^2_{j, \lambda_j}\left(x_j\right) \dx{\mu^{(j)}}(x) = \prod_{j=1}^d \dx{\mu}^{(j)}_{\lambda_j},
\end{align*}
where $\dx{\mu}^{(j)}_{\lambda_j}$ is the (univariate) order-$\lambda_j$ induced measure for $\mu^{(j)}$ according to the definition \eqref{eq:mun-def}. Thus, $\mu_\lambda$ is also a tensorial measure. %We will also use the notation $F_\lambda$ to denote the distribution function for $\mu_\lambda$.

\subsection{Optimal polynomial discrete least-squares}\label{sec:applications-ls}
The goal of this section is description of a procedure utilizing the algorithms above for performing discrete least-squares recovery in a polynomial subspace using the optimal (fewest) number of samples. The procedure we discuss was proposed in \cite{cohen_optimal_2016} and is based on the foundational matrix concentration estimates for least-squares derived in \cite{cohen_stability_2013}. 

Let $f: \R^d \rightarrow \R$ be a $d$-variate function. Given (i) a tensorial probability measure $\mu$ admitting an orthonormal polynomial family, and (ii) a dimension-$N$ polynomial subspace $P_\Lambda$, we are interested in approximating the $L^2_{\dx{\mu}}$-orthogonal projection of $f$ onto $P_\Lambda$. This projection is given explicitly by 
\begin{align*}
  \Pi_\Lambda f &= \sum_{\lambda \in \Lambda} c^\ast_\lambda p_\lambda(x), & c^\ast_\lambda &= \int_{\R^d} f(x) p_\lambda(x) \dx{\mu}(x).
\end{align*}
One way to approximate the integral defining the coefficients $c^\ast_\lambda$ is via a Monte Carlo least-squares procedure using $M$ collocation samples of the function $f(x)$. Let $\left\{X_m\right\}_{m=1}^M$ denote a collection of $M$ independent and identically distributed random variables on $\R^d$, where we leave the distribution of $X_m$ unspecified for the moment. A weighted discrete least-squares recovery procedure approximates $c^\ast_\lambda$ with $c_\lambda$, computed as
\begin{align}\label{eq:discrete-least-squares}
  \left\{c_\lambda \right\}_{\lambda \in \Lambda} = \argmin_{d_\lambda \in \R} \frac{1}{M} \sum_{m=1}^M w_m \left[ f(X_m) - \sum_{\lambda \in \Lambda} d_\lambda p_\lambda(X_m)\right]^2,
\end{align}
where $w_m$ are positive weights. One supposes that if the distribution of $X_m$ and the weights $w_m$ are chosen intelligently, then it is possible to recover the $N$ coefficients $c_\lambda$ with a relatively small number of samples $M$; ideally, $M$ should be close to $N$. The analysis in \cite{cohen_stability_2013} codifies conditions on a required sample count $M$ so that the minimization procedure above is stable, and so that the recovered coefficients $c_\lambda$ are ``close" to $c_\lambda^\ast$; these conditions depend on the distribution of $X_m$, on $w_m$, on $\mu$, and on $P_\Lambda$. Since $\mu$ and $P_\Lambda$ are specified, the goal here is identification of an appropriate distribution for $X_m$ and weight $w_m$.

Using ideas proposed in \cite{narayan_christoffel_2016,hampton_coherence_2015} the results in \cite{cohen_optimal_2016} show that, in the context of the analysis in \cite{cohen_stability_2013}, the \textit{optimal} choice of probability measure $\mu_X$ for sampling $X_m$ and weights $w_m$ that achieves a minimal sample count $M$ are
\begin{align}\label{eq:optimal-sampling}
  \mu_X = \mu_\Lambda &= \frac{1}{N} \sum_{\lambda \in \Lambda} \mu_\lambda, & w_m &= \frac{N}{\sum_{\lambda \in \Lambda} p_\lambda^2\left(X_m\right)}.
\end{align}
The precise quantification of the sample count and error estimates can be formulated using an algebraic characterization of \eqref{eq:discrete-least-squares}. Define matrices $\bs{V} \in \R^{M \times N}$ and $\bs{W} \in \R^{M \times M}$, and vectors $\bs{c} \in \R^N$ and $\bs{f} \in \R^M$ as follows:
\begin{align*}
  (V)_{m,n} &= p_{\lambda(n)}\left(X_m\right), & (W)_{j,k} &= w_j \delta_{j,k}, \\
  (c)_{n} &= c_{\lambda(n)}, & (f)_{m} &= f\left(X_m\right),
\end{align*}
where $\lambda(1), \ldots, \lambda(N)$ represents any enumeration of elements in $\Lambda$. We use $\|\cdot\|$ on matrices to denote the induced $\ell^2$ norm. The algebraic version of \eqref{eq:discrete-least-squares} is then to compute $\bs{c}$ that minimizes the the least-squares residual of $\sqrt{\bs{W}} \bs{V} \bs{c} = \sqrt{\bs{W}} \bs{f}$. The following result holds.
\begin{theorem}[\cite{cohen_stability_2013,cohen_optimal_2016}]
  Let $0 < \delta < 1$,  and $r > 0$ be given, and define $c_\delta \coloneqq \delta + (1-\delta) \log(1-\delta) \in (0,1)$. Draw $M$ iid samples $\left\{ X_m\right\}_{m=1}^M$ from $\mu_X$, and let the coefficients $c_\lambda$ be those recovered from \eqref{eq:discrete-least-squares}. If
  \begin{align}\label{eq:ls-sample-count}
    \frac{M}{\log M} \geq N \frac{1+r}{c_\delta}
  \end{align}
  Then 
  \begin{align*}
    \mathrm{Pr} \left[ \left\| \bs{V}^T \bs{W} \bs{V} - \bs{I} \right\| > \delta \right] &\leq 2 M^{-r} \\
    \E \left\| f - T_L\left(\sum_{\lambda \in \Lambda} c_\lambda p_\lambda(\cdot)\right) \right\|_{L^2_{\dx{\mu}}} &\leq \left[ 1 + \frac{4c_\delta}{(1+r) \log M} \right] \left\| f - \Pi_\Lambda f \right\|_{L^2_{\dx{\mu}}} + 8 \left\|f\right\|_{L^\infty(\supp \mu)} M^{-r}
  \end{align*}
\end{theorem}
The free parameter $r$ is a tunable oversampling rate; $\delta$ represents the guaranteeable proximity of $\bs{V}^T \bs{W} \bs{V}$ to $\bs{I}$. We emphasize that by choosing $\mu_X = \mu_\Lambda$ with the weights defined as in $\bs{W}$, then the size of $M$ as required by \eqref{eq:ls-sample-count} depends \textit{only} on the the cardinality $N$ of $\Lambda$, and not on its shape. Furthermore, the criterion $M/\log M \gtrsim N$ is optimal up to the logarithmic factor. Also, the statements above hold uniformly over all multivariate $\mu$. 

Note that the optimal sampling measure $\mu_X$ is an additive mixture of induced measures and can be easily sampled, assuming $\mu_\lambda$ can be sampled. Sampling from $\mu_X$ defined above is fairly straightforward given the algorithms in this paper: (1) given $\Lambda$ choose an element $\lambda$ randomly using the uniform probability law, (2) generate $d$ independent, uniform, continuous random variables $U_j$, $j=1, \ldots, d$ each on the interval $[0,1]$, (3) compute $X \in \R^d$ as
\begin{align*}
  X = \left( F^{-1}_{\lambda_1}\left(U_1; \mu^{(1)}\right), \;\;\; F^{-1}_{\lambda_2}\left(U_2; \mu^{(2)}\right), \;\;\; \ldots, \;\;\; F^{-1}_{\lambda_d}\left(U_d; \mu^{(d)}\right) \right).
\end{align*}
Then $X$ is a sample from the probability measure $\mu_X$. Note that the work required to sample $X$ requires only $d$ samples from a univariate induced measure. The procedure above is essentially as described by the authors in \cite{cohen_optimal_2016}; this paper gives a concrete computational method to sample from $\mu_\Lambda$ for a relatively general class of measures $\mu$ (i.e., those formed by arbitrary finite tensor products of Jacobi, half-line Freud, and/or Freud univariate measures). Thus, the algorithms in this paper along with the specifications \eqref{eq:optimal-sampling} allow one to perform optimal discrete least-squares using Monte Carlo sampling for approximation with multivariate polynomials.

\subsection{Weighted equilibrium measures}\label{sec:applications-equilibrium}

On $\R^d$, consider the special case $\dx{\mu}(x) = \exp(-\|x\|^2)$, where $\|\cdot\|$ is the Euclidean norm on $\R^d$. The weighted equilibrium measure $\mu^\ast$ is a probability measure that is the weak limit of the summations
\begin{align}\label{eq:hermite-limit}
  \lim_{n \rightarrow \infty} \sum_{|\lambda| \leq n} p_\lambda^2\left(x/\sqrt{2n}\right) \dx{\mu}(x/\sqrt{2n}) \Rightarrow \dx{\mu^\ast}(x).
\end{align}
The form for $\mu^\ast$ is not currently known, but the authors in \cite{narayan_christoffel_2016} conjecture that $\mu^\ast$ has support on the unit ball with density
\begin{align}\label{eq:conjecture}
  \dx{\mu^\ast}(x) &\stackrel{?}{=} g_d(\left\| x\right\|) = C_d \left(1 - \left\|x\right\|^2\right)^{d/2}, & \left\| x \right\| \leq 1,
\end{align}
where $C_d = \left(\pi\right)^{-(d+1)/2} \Gamma\left(\frac{d+1}{2}\right)$. If $X$ on $\R^d$ is distributed according to $g_d$, then the cumulative distribution function associated to $\left\|X \right\|$ is 
\begin{align}\label{eq:equilibrium-conjecture-distribution}
  G_d(r) \coloneqq Pr\left[ \left\| X \right\| \leq r \right] = K \int_0^r g_d(x) r^{d-1} \dx{x},
\end{align}
where the $r^{d-1}$ factor in the integrand is the $\R^d$ Jacobian factor for integration in spherical coordinates, and $K$ is the associated normalization constant. Note that the cumulative distribution function $G_d$ is a mapped (normalized) incomplete Beta function with parameters $a = d/2$ and $b = 1 + d/2$,  
\begin{align*}
  G_d(r) &= \frac{1}{B\left(\frac{d}{2}, 1 + \frac{d}{2} \right)} \int_0^{r^2} t^{d/2} (1-t)^{1 + d/2} \dx{t},
\end{align*}
where $B(\cdot,\cdot)$ is the Beta function. With $d=1$, the veracity of this limit is known \cite{}. Using the algorithms in this paper, we can empirically test the conjecture. Precisely, defining $\Lambda_n \coloneqq \left\{ \lambda \in \N_0^d \; | \; |\lambda| \leq n \right\}$, then the conjecture for \eqref{eq:hermite-limit} reads
\begin{align*}
  \lim_{n \rightarrow \infty} \sum_{|\lambda| \leq n} p_\lambda^2\left(x/\sqrt{2n}\right) \dx{\mu}(x/\sqrt{2n}) = \lim_{n \rightarrow \infty} \mu_{\Lambda_n}(x/\sqrt{2n}) \stackrel{?}{\Rightarrow} C_d \left( 1 - \left\| x\right\|^2 \right)^{d/2}
\end{align*}
Our procedure for testing this conjecture is as follows: for a fixed $d$ and large $n$, we generate $M$ iid samples $\left\{ X_m \right\}_{m=1}^M$ distributed according to $\mu_{\Lambda_n}$, and compute the empirical distribution function associated with the ensemble of scalars $\left\{ \left\| X_m \right\|/\sqrt{2 n} \right\}_{m=1}^M$. We show in Figure \ref{fig:equilibrium} that indeed for large $n$ that empirical distributions associated with these ensembles match very closely with the distribution function $G_d(r)$, giving evidence that supports, but does not prove, the conjecture \eqref{eq:conjecture}.

\begin{figure}
  \begin{center}
    \includegraphics[width=0.32\textwidth]{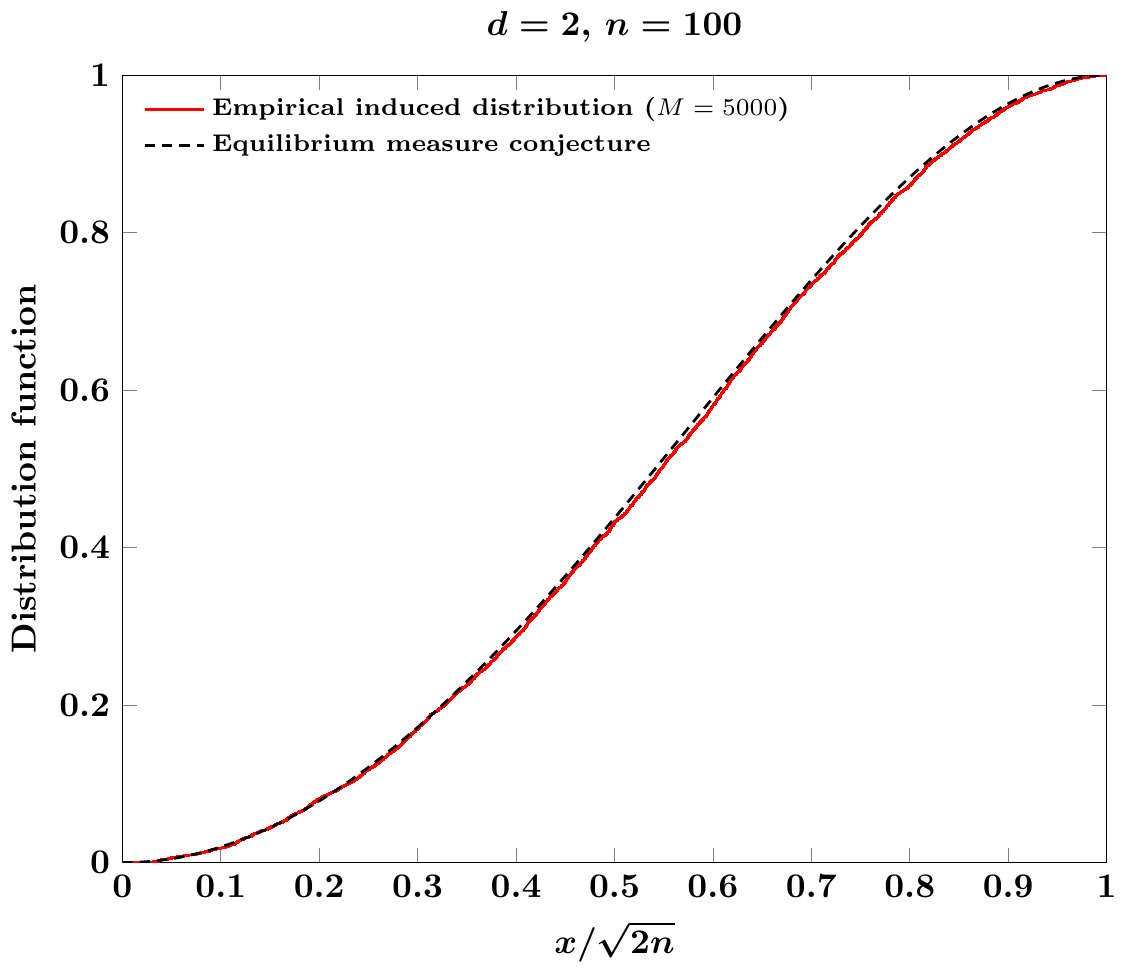}
    \includegraphics[width=0.32\textwidth]{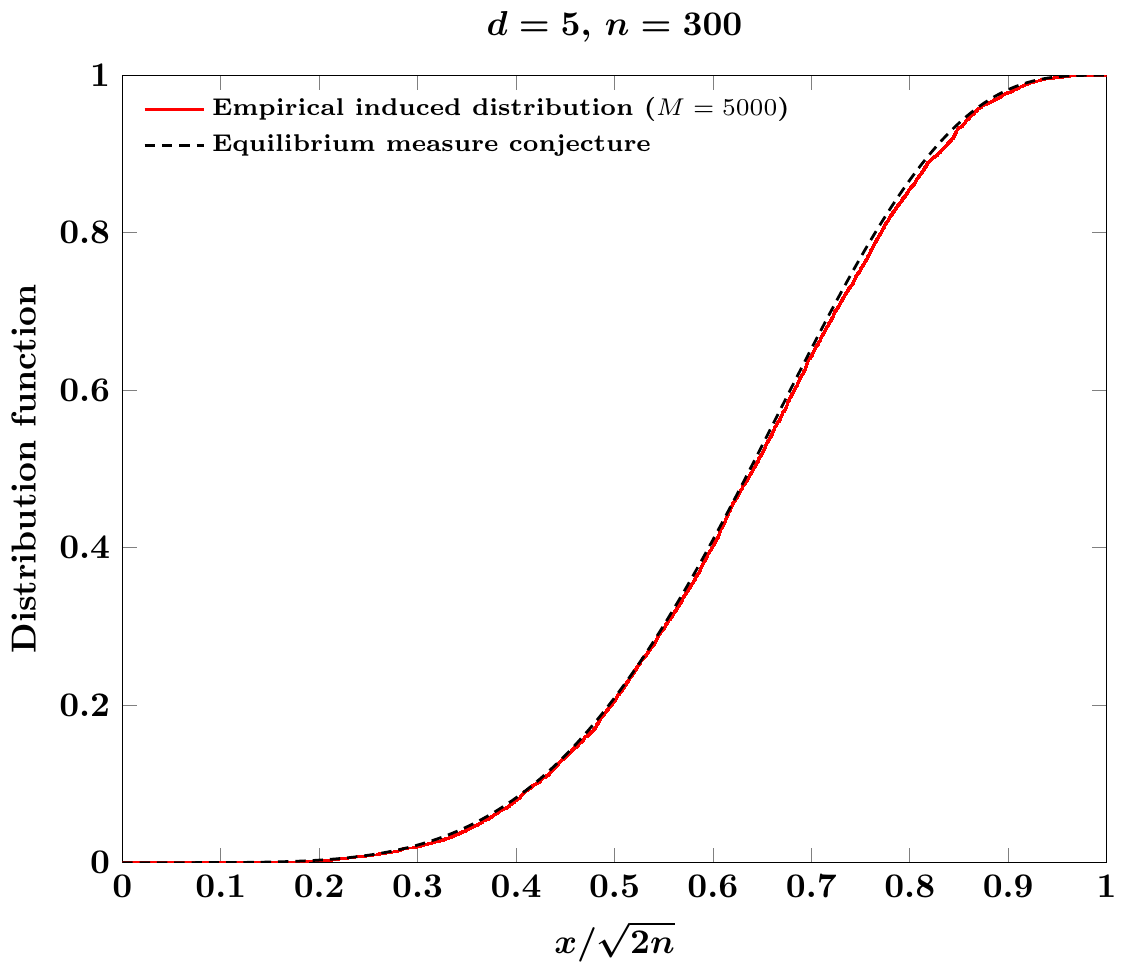}
    \includegraphics[width=0.32\textwidth]{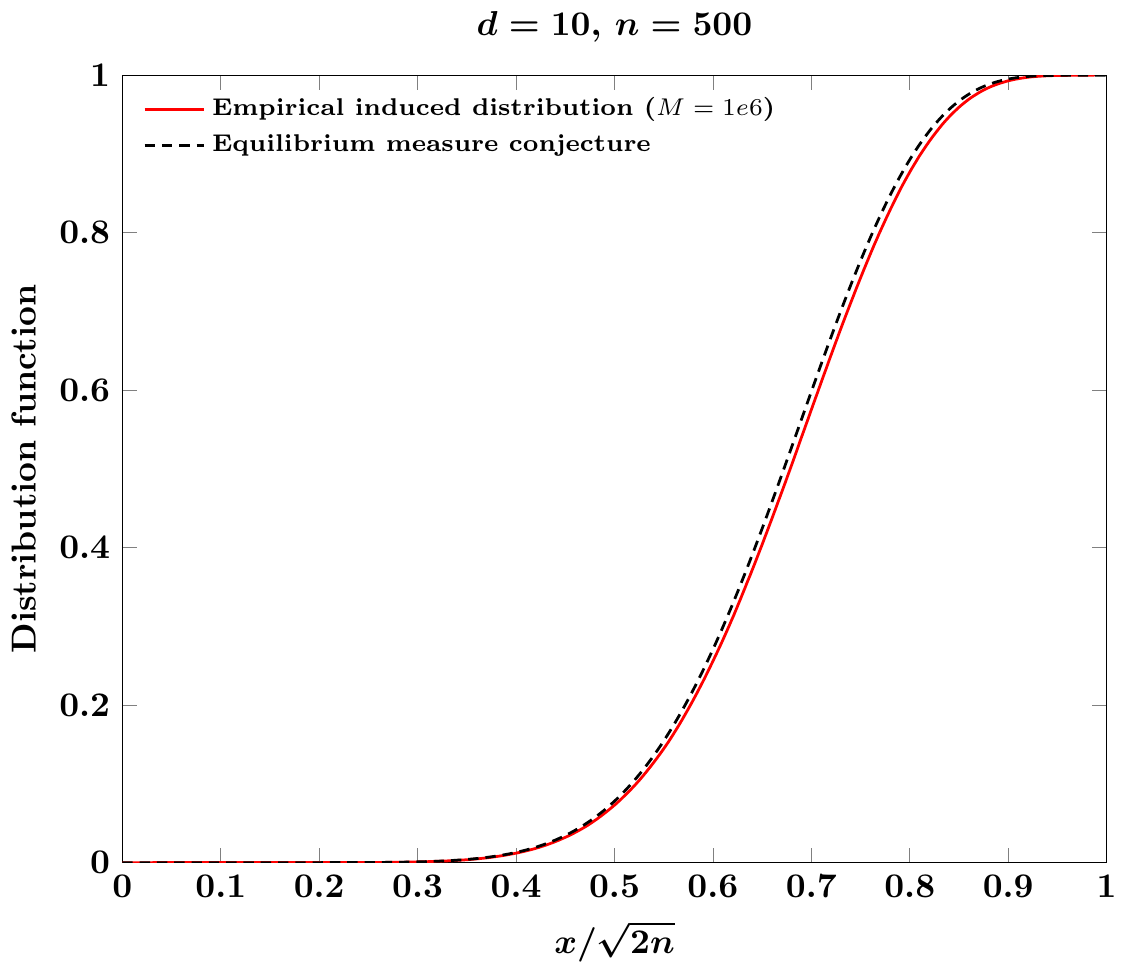}
  \end{center}
  \caption{Weighted equilibrium measure distribution conjecture \eqref{eq:equilibrium-conjecture-distribution} versus empirical distribution function of $M$ iid samples of $\|X\|/\sqrt{2n}$, where $X$ is drawn from the induced measure $\mu_{\Lambda_n}$. Left: $(d,n) = (2, 100)$. Middle: $(d,n) = (5, 300)$. Right: $(d,n) = (10, 500)$. That these distribution functions visually match gives credence to the conjecture \eqref{eq:conjecture} first formulated in \cite{narayan_christoffel_2016}.}\label{fig:equilibrium}
\end{figure}

\section{Conclusions}
We have developed a robust algorithm for the evaluation of induced polynomial distribution functions associated with a relatively wide class of continuous univariate measures. Our algorithms cover all classical orthogonal polynomial measures, and are equally applicable on bounded or unbounded domains. The algorithm leverages several properties of orthogonal polynomials in order to attain stability and accuracy, even for extremely large values of parameters defining the measure or polynomial degree. All computations have been tested up to degree $n=1000$ and were found to be stable. The ability to evaluate induced distributions allows the possibility to exactly sample from additive mixtures of these measures. Such additive mixtures define sampling densities that are known to be optimal for multivariate discrete least-squares polynomial approximation algorithms, and allow us to provide supporting empirical evidence for an asymptotic conjecture involving weighted pluripotential equilibrium measures.

\appendix

\section{Auxiliary recurrences}\label{app:auxilliary}

For some algorithmic tasks that we consider, the three-term recurrence \eqref{eq:three-term-recurrence} for the $p_n$ does not provide a suitable computational procedure due to floating-point under- and over-flow. This happens in two particular cases:
\begin{itemize}
  \item If $x$ is far outside $\supp{} \mu$, then $p_n(x)$ becomes very large and causes numerical overflow (the quantity grows like $x^n$). We will need to evaluate $p_n(x)/p_{n-1}(x)$ for large $x$ and potentially large $n$. (When $\supp{} \mu$ is infinite, one can interpret ``far outside $\supp{} \mu$" to be defined using the potential-theoretic Mhaskar-Rakhmanov-Saff numbers for $\sqrt{\dx{\mu}(x)}$.)
  \item When $x$ is inside $\supp \mu$, we will need to evaluate $p_n^2(x)/\sum_{j=0}^{n-1} p_j^2(x)$. For large enough $n$, a direct computation causes numerical overflow.
\end{itemize}
We emphasize that \eqref{eq:three-term-recurrence} is quite stable and sufficient for most practical computations requiring evaluations of orthogonal polynomials. The situations we describe above (which occur in this paper) are relatively pathological. 

\subsection{Ratio evaluations}\label{app:auxilliary:ratio}
We consider the first case described above. With $n$ fixed, suppose that either $x > \max p_{n-1}^{-1}(0)$, or $x < \min p_{n-1}^{-1}(0)$. Then by the interlacing properties of orthogonal polynomial zeros, $p_{j}(x) \neq 0$ for all $j=0, \ldots n-1$. In this case, the ratio
\begin{align}\label{eq:ratio-def}
  r_j(x) &\coloneqq \frac{p_j(x)}{p_{j-1}(x)}, & 1 \leq j < n,
\end{align}
is well-defined, with $r_0 \coloneqq p_0$. A straightforward manipulation of \eqref{eq:three-term-recurrence} yields
\begin{align}\label{eq:ratio-ttr}
  \sqrt{b_j} r_j(x) &= x - a_j - \frac{\sqrt{b_{j-1}}}{r_{j-1}(x)}, & 1 \leq j < n.
\end{align}
The recurrence \eqref{eq:ratio-ttr} is a more stable way to compute $r_j(x)$ when $x$ is very large. In practice we can computationally verify that $x$ lies outside the zero set of $p_{n-1}$ with $\mathcal{O}(n)$ effort (e.g., \cite[equation (11)]{nevai_asymptotic_1979} for a crude but general estimate). In the context of this paper, this condition is always satisfied whenever we require an evaluation of $r_n(x)$. 

\subsection{Normalized polynomials}\label{app:auxilliary:quadratic}
In the second case, consider a different normalization of $p_n$:
\begin{align}\label{eq:C-def}
  C_n(x) &\coloneqq \frac{p_n(x)}{\sqrt{\sum_{j=0}^{n-1} p_j^2(x)}}, & n &> 0, \; x\in \R %\\
\end{align}
with $C_0 \equiv p_0 = \sqrt{1/b_0}$. A manipulation of the three-term recurrence relation \eqref{eq:three-term-recurrence} yields the following recurrence for $C_n$:
\begin{subequations}\label{eq:C-recurrence}
\begin{align}
  C_0(x) &= \frac{1}{\sqrt{b_{0}}}, \\
  C_1(x) &= \frac{1}{\sqrt{b_{1}}} \left( x - a_0 \right), \\
  C_2(x) &= \frac{1}{\sqrt{b_{2}}\sqrt{1 + C_1^2(x)}} \left[ (x - a_1) C_1(x) - \sqrt{b_1} \right] \\
  C_{n+1}(x) &= \frac{1}{\sqrt{b_{n+1}}\sqrt{1 + C_n^2(x)}} \left[ (x - a_n) C_n(x) - \sqrt{b_n} \frac{C_{n-1}(x)}{\sqrt{1 + C_{n-1}^2(x)}}\right], \hskip 10pt n \geq 2
\end{align}
\end{subequations}
Note that $C_n(x)$ essentially behaves like $r_n$ outside a compact interval containing the zero set of $p_n$; however, $C_n$ is well-defined and well-behaved inside this compact interval, unlike $r_n$. The polynomials $p_n$ may be reproduced from knowledge of $C_n$:
\begin{align*}
  p_n(x) &= C_0 C_n(x) \prod_{j=1}^{n-1} \sqrt{1 + C_j^2(x)}, & n &> 0.
\end{align*}

\section{Polynomial measure modifications}\label{app:measure-modifications}

We will need to compute recurrence coefficients for the modified measures with densities
\begin{align*}
  \dx{\widetilde{\mu}}(x) &= \pm \left(x - y_0 \right) \dx{\mu}(x), & y_0 &\not\in \mathrm{supp}\;\mu\\
  \dx{\widetilde{\widetilde{\mu}}}(x) &= \left(x - z_0 \right)^2 \dx{\mu}(x), & z_0 &\in \R
\end{align*}
where we assume that the recurrence coefficients of $\mu$ are available to us. Here, both $y_0$ and $z_0$ are some fixed real-valued numbers. In the first case (a linear modification) we assume $y_0 \not\in \supp \mu$ and choose the sign to ensure that $\dx{\widetilde{\mu}}(x)$ is positive for $x \in \supp \mu$. Assuming the recurrence coefficients $a_n$ and $b_n$ for $\mu$ are known, the problems of computing the recurrence coefficients $\tilde{a}_n$ and $\tilde{b}_n$ for $\widetilde{\mu}$, and of computing the recurrence coefficients $\widetilde{\widetilde{a}}_n$ and $\widetilde{\widetilde{b}}_n$ for $\widetilde{\widetilde{\mu}}$, are well-studied and have constructive computational solutions \cite{golub_modified_1973}.

We use the auxiliary variables defined in Appendix \ref{app:auxilliary} to accomplish measure modifications. The linear and quadratic modification recurrence coefficients have the following forms (cf. \cite[Section 4]{narayan_computation_2012}):
\begin{subequations}\label{eq:modifications}
\begin{align}\label{eq:linear-modification}
  \widetilde{a}_n &= a_n + \widetilde{\Delta a}_n, & \widetilde{b}_n &= b_n \widetilde{\Delta b}_n, \\\label{eq:quadratic-modification}
  \widetilde{\widetilde{a}}_n &= a_{n+1} + \widetilde{\widetilde{\Delta a}}_n, & \widetilde{\widetilde{b}}_n &= b_{n+1} \widetilde{\widetilde{\Delta b}}_n
\end{align}
\end{subequations}
The correction factors for $n > 0$ are given by
\begin{align*}
  \widetilde{\Delta a}_n &= \frac{\sqrt{b_{n+1}}}{{r_{n+1}\left(y_0\right)}} - \frac{\sqrt{b_{n}}}{{r_n\left(y_0\right)}}, & \widetilde{\Delta b}_n &= \frac{\sqrt{b_{n+1}} r_{n+1}\left(y_0\right)}{\sqrt{b_n} r_n\left(y_0\right)}, \\
  \widetilde{\widetilde{\Delta a}}_n &= \sqrt{b_{n+2}} \frac{ C_{n+2}\left(y_0\right) C_{n+1}\left(y_0\right)}{\sqrt{1 + C_{n+1}^2\left(y_0\right)}} - \sqrt{b_{n+1}} \frac{ C_{n+1}\left(y_0\right) C_{n}\left(y_0\right)}{\sqrt{1 + C_{n}^2\left(y_0\right)}}, &
  \widetilde{\widetilde{\Delta b}}_n &= \frac{1 + C^2_{n+1}\left(y_0\right)}{1 + C_{n}^2\left(y_0\right)}
\end{align*}
For $n=0$ they take the special forms
\begin{align*}
  \widetilde{\Delta a}_0 &= \frac{\sqrt{b_{1}}}{{r_{1}\left(y_0\right)}}, & \widetilde{\Delta b}_0 &= \sqrt{b_{1}} r_{1}\left(y_0\right),\\
  \widetilde{\widetilde{\Delta a}}_0 &= \sqrt{b_{2}} \frac{ C_{2}\left(y_0\right) C_1\left(y_0\right)}{\sqrt{1 + C_{1}^2\left(y_0\right)}} - \sqrt{b_1} \frac{ C_1\left(y_0\right) C_{0}\left(y_0\right)}{\sqrt{1 + C_{0}^2\left(y_0\right)}}, &
  \widetilde{\widetilde{\Delta b}}_0 &= \frac{1+C_1^2(y_0)}{C_0^2}
\end{align*}
Above, $r_n(x) = r_n\left(x; \mu\right)$ and $C_n(x) = C_n\left(x; \mu\right)$ are the functions associated with the measure $\mu$ and so may be readily evaluated using \eqref{eq:ratio-ttr} and \eqref{eq:C-recurrence}. 

Note that if we only have a finite number of recurrence coefficients, $\left\{a_n, b_n \right\}_{n=0}^M$ for $\mu$, then a linear modification can only compute modified coefficients up to index $M-1$, and a quadratic modification can only compute coefficients up to index $M-2$.

\section{Freud and half-line Freud recurrence coefficients}\label{app:recurrence-coefficients}
For both cases of Freud measures with $\alpha = 2$ (generalized Hermite polynomials), and generalized Freud measures with $\alpha = 1$ (generalized Laguerre polynomials), explicit forms for the recurrence coefficients are known. However, the situation is more complicated for other values of $\alpha$. 

We give an extension of Lemma \ref{lemma:hermite-laguerre-relation}: Recurrence coefficients of generalized Freud weights may be computed from those of Freud weights.
\begin{lemma}
  Let parameters $(\alpha,\rho)$ define a Freud weight having recurrence coefficients $\left\{b_n\right\}_{n=0}^\infty$. (The $a_n$ coefficients vanish because the weight is even.) Define $(\alpha_\ast,\rho_\ast)$ and $(\alpha_{\ast\ast}, \rho_{\ast\ast})$ as in \eqref{eq:all-ast-def}, along with the associated generalized Freud measures $\mu_{F\ast}^{(\alpha_\ast,\rho_\ast)}$ and $\mu_{F\ast}^{(\alpha_{\ast\ast}, \rho_{\ast\ast})}$ and their recurrence coefficients $\left\{(a_{\ast,n}, b_{\ast,n})\right\}_{n=0}^\infty$ and $\left\{(a_{\ast\ast,n}, b_{\ast\ast,n})\right\}_{n=0}^\infty$, respectively. Then, for all $n$:
  \begin{subequations}
    \begin{align}\label{eq:ab0}
      a_{\ast,0} &= b_1, & b_{\ast,0} &= b_{0}, & & \\
      a_{\ast,n} &= b_{2 n} + b_{2 n + 1}, & b_{\ast,n} &= b_{2 n} b_{2 n -1}, & n &\geq 1
    \end{align}
  \end{subequations}
  Furthermore, 
  \begin{subequations}
    \begin{align}
      b_0 &= b_{\ast,0}, & b_1 &= a_{\ast,0}, & & \\
      b_{2n} &= b_{\ast,n}/b_{2n-1}, & b_{2n+1} &= a_{\ast,n} - b_{2n}, & n &\geq 1
    \end{align}
  \end{subequations}
\end{lemma}
This result implies that one may either use Freud weight recurrence coefficients to compute Half-line Freud weight recurrence coefficients, or vice versa. The proof is a result of Lemma \ref{lemma:hermite-laguerre-relation} along with manipulations of the three-term recurrence relation \eqref{eq:three-term-recurrence}.

\bibliographystyle{siamplain}
\bibliography{references}

\end{document}